\documentclass[10pt]{amsart}
\usepackage{amssymb}
\usepackage{amsmath}
\usepackage{tikz}
\usepackage{hyperref}
\hypersetup{
	pdfstartview={XYZ null null 1.00}, 
	pdfpagemode=UseNone, 
	colorlinks,
	breaklinks, 
	linkcolor=blue,
	urlcolor=blue, 
	anchorcolor=blue,
	citecolor=blue
}

\usepackage{comment}

\newtheorem{theorem}{Theorem}[section]
\newtheorem{question}[theorem]{Question}
\newtheorem{lemma}[theorem]{Lemma}
\newtheorem{prop}[theorem]{Proposition}
\newtheorem{cor}[theorem]{Corollary}

\newtheorem{claim}[theorem]{Claim}
\newtheorem{example}[theorem]{Example}
\newtheorem{proposition}[theorem]{Proposition}

\theoremstyle{definition}
\newtheorem{definition}[theorem]{Definition}
\newtheorem{remark}[theorem]{Remark}

\newcommand{\zz}{(\mathbb{Z}\slash2\mathbb{Z}) *(\mathbb{Z}\slash2\mathbb{Z})}

\DeclareMathOperator{\supp}{supp}

\numberwithin{equation}{section}

\begin{document}
\title{Perfect matchings in hyperfinite graphings}

\author{Matthew Bowen}

\address{Department of Mathematics and Statistics, McGill University, 805 Sherbrooke St W., H3A 0B9  Montreal, Canada}

\email{ matthew.bowen2@mail.mcgill.ca}
\author{G\'abor Kun}

\address{Alfr\'ed R\'enyi Institute of Mathematics, H-1053 Budapest, Re\'altanoda u. 13-15., Hungary}
\address{Institute of Mathematics, E\"otv\"os L\'or\'and University, P\'azm\'any P\'eter s\'et\'any 1/c, H-1117 Budapest, Hungary}

\email{kungabor@renyi.hu}

\author{Marcin Sabok}

\address{Department of Mathematics and Statistics, McGill University, 805 Sherbrooke St W., H3A 0B9  Montreal, Canada}
\email{marcin.sabok@mcgill.ca}

\thanks{The first and the third authors are partly funded by the NSERC Discovery Grant  RGPIN-2020-05445, NSERC Discovery Accelerator
Supplement  RGPAS-2020-00097 and NCN Grant Harmonia  2018/30/M/ST1/00668. The second author's work on the project leading to this application has received funding from the European Research Council (ERC) under the European Union's Horizon 2020 research and innovation programme (grant agreement No. 741420), from the \'UNKP-20-5 New National Excellence Program of the Ministry of Innovation and Technology from the source of the National Research, Development and Innovation Fund and from the J\'anos Bolyai Scholarship of the Hungarian Academy of Sciences.}


\begin{abstract}

We characterize hyperfinite bipartite graphings that admit measurable perfect matchings. In particular, we prove that every regular hyperfinite bipartite graphing admits a measurable perfect matching if it is one-ended or the degree is odd. We give several applications of this result, answering various open questions in the field.
For instance, we extend the Lyons--Nazarov theorem by characterizing bipartite Cayley graphs which admit a factor of iid perfect matching, 
answering the bipartite case of a well-known question of Lyons and Nazarov 
popularized by Kechris and Marks
.
Moreover, we show how our results apply to measurable equidecompositions and, in particular,  generalize the recent result of Grabowski, M\'ath\'e and Pikhurko on the measurable circle squaring. Our approach applies more generally to rounding measurable perfect fractional matchings. 
\end{abstract}
\maketitle


\section{Introduction}

Measurable 
combinatorics is concerned with
combinatorial problems arising in the setting of probability
measure preserving (pmp) Borel actions of
finitely generated groups, and more generally
\textit{graphings}, i.e., Borel pmp graphs on standard
Borel spaces (see the textbook of Lov\'asz \cite{lovasz}).  In this paper we study
hyperfinite graphings (see the work of Elek \cite{elek} and Schramm \cite{schramm}) and, in particular,
pmp actions of finitely generated amenable groups.

 Measurable perfect matchings have received a considerable
amount of attention recently (see for instance the
survey of Kechris and Marks \cite{kechris.marks}). In this paper we work with \textit{bipartite} graphs, i.e. graphs which admit no odd cycles. 
A \textit{perfect fractional matching} on a graph $(V,E)$ is a function $\varphi:E\to[0,1]$ such that for every vertex $v$ we have $\sum_{v\in e} \varphi(e)=1$.
Note that every $d$-regular graph admits a perfect fractional matching given by the constant function $\varphi(e)=\frac{1}{d}$ . For finite bipartite graphs, both the existence of a perfect matching and the existence of a perfect fractional matching are equivalent to Hall's condition.
In the measurable context, this is not true.
This can be deduced from the Banach--Tarski paradox \cite{banach-tarski},
which gives a graphing 
that does not admit a measurable perfect fractional
matching (for more on this topic see \cite{kechris.marks}). Even in the hyperfinite
case, Laczkovich \cite{lacz3} gave an example of a $2$-regular 
acyclic graphing that admits  no  measurable  perfect  matching. 
 

Recall that in a hyperfinite graphing, every infinite
connected component has either one or two ends (see the paper of Adams
\cite{adams} or the textbook of Kechris and Miller \cite[Theorem 22.3]{kechris.miller}). 
Note that a hyperfinite graphing is a.e. one-ended if
and only if the growth of balls is a.e. superlinear (see
Proposition \ref{equi}).

For hyperfinite bipartite graphings the existence of a perfect matching is equivalent to 
the existence of a measurable perfect fractional matching (see Lemma \ref{lemat}). The example of Laczkovich \cite{lacz3} shows that in the two-ended case, this does not imply the existence of a measurable perfect matching.  As it turns out, even in the one-ended  bipartite case the existence of a measurable perfect fractional matching does not imply the existence of a measurable perfect matching (see Example \ref{example}). Our main result states that the existence of a measurable perfect matching follows in the one-ended  bipartite case when a perfect fractional matching is a.e. positive.

\begin{theorem}\label{matching}
Let $G$ be a hyperfinite one-ended bipartite graphing. If $G$ admits a measurable perfect fractional matching which is everywhere positive, then $G$ admits a measurable perfect matching.
\end{theorem}

More generally, in Theorem \ref{fmatching} we show that every measurable perfect fractional matching with hyperfinite, one-ended support is in the closure of convex combinations of measurable perfect matchings --- even though, unlike in the finite case, not every extreme point of the set of perfect fractional matchings must be integral. 

Theorem \ref{matching} has been recently applied by Tim\'ar \cite{timar.new} to construct a factor matching of optimal tail between two Poisson point processes in $\mathbb{R}^n$. This provides the optimal decay that a Poisson factor matching can achieve, improves previous results and answers the problem initiated by Holroyd, Pemantle, Peres and Schramm \cite{poisson.matching}. 

As a special case of Theorem \ref{matching}, we get measurable perfect matchings when the graphing is regular.

\begin{cor}\label{regular}
Every regular hyperfinite one-ended bipartite graphing admits a measurable perfect matching.
\end{cor}

We show that for regular hyperfinite bipartite graphings of odd degree the assumption on one-endedness is not needed. We give a short proof of the following.

\begin{theorem}\label{odd}
Every regular hyperfinite bipartite graphing of odd degree admits a measurable perfect matching.
\end{theorem}
This result is optimal, as for every even number $d$ there exist a $d$-regular hyperfinite graphing, (obtained by a modification of  \cite{lacz3}), which does not admit measurable perfect matchings \cite{klopotowski, conley.kechris}. In \cite[Problem 13.6]{kechris.marks} Kechris and Marks asked about the existence of measurable perfect matchings in $3$-regular Borel graphs. While the answer for general bipartite graphings (even in the acyclic case) is negative \cite{gabor.new}, Theorem \ref{odd} shows that in case of hyperfinite bipartite graphings the answer is positive. 

The regular case for graphings of arbitrary degree applies in particular to the bipartite Cayley (or Schreier) graphs
induced by actions of finitely generated groups.
Lyons and Nazarov \cite{lyons.nazarov} proved that if the Cayley graph of a finitely generated non-amenable group is
bipartite then it admits a.s. a factor of iid perfect matching. Gao,
Jackson, Krohne and Seward  showed that for $n \geq 2$ and the standard set of generators of $\mathbb{Z}^d$, the Schreier graph of the 
shift of $\mathbb{Z}^n$ on $\{0,1\}^{\mathbb{Z}^n}$ admits a Borel perfect matching on the free part of the shift (cf. \cite[Theorem 10.2]{kechris.marks}). More generally, finitely generated abelian groups which admit a measurable perfect matching were recently characterized by Weilacher \cite[Theorem 2]{weilacher}.

 In \cite{lyons.nazarov} Lyons and Nazarov asked which Cayley 
graphs admit a.s. a factor of iid perfect matching. The question was popularized by Kechris and Marks in their survey \cite[Question 13.5]{kechris.marks}. We use Corollary \ref{regular} to answer that question for bipartite Cayley graphs, extending the
Lyons--Nazarov theorem. 




\begin{theorem}\label{fiid}
Let $\Gamma$ be a finitely generated group. 
\begin{itemize}
    \item If $\Gamma$ is isomorphic to $\mathbb{Z} \ltimes \Delta$ for a finite normal subgroup $\Delta$ of odd order, then no bipartite Cayley graph of $\Gamma$ admits a factor of iid perfect matching a.s.

    \item Else, if $\Gamma$ is not isomorphic to $\mathbb{Z} \ltimes \Delta$ 
    for a finite normal subgroup $\Delta$ of odd order, then every bipartite Cayley graph of $\Gamma$ admits a factor of iid perfect matching a.s.
\end{itemize}
\end{theorem}


In Section \ref{sec:rounding} we give more general theorems on rounding perfect fractional matchings in hyperfinite graphings. In Corollary \ref{characterization} we  characterize hyperfinite bipartite graphings with measurable perfect matchings by reducing the problem to the two-ended case. In Corollaries \ref{subgraph toth} and \ref{balanced} we also find a regular spanning subgraphing of arbitrary degree less than the degree of the graph, in regular hyperfinite one-ended bipartite graphings, and so-called balanced orientations of one-ended graphings with even degree, answering a question of Bencs, Hru\v{s}kov\'a and T\'oth \cite[Question 6.4]{bht}.


In the proofs  we use particular witnesses to hyperfiniteness in one-ended hyperfinite graphings, which we call a connected toast structure.  This relies on the work of Conley, Gaboriau, Marks, Tucker-Drob \cite{cgmtd} and Tim\'ar \cite{timar} (extending the earlier work of Benjamini, Lyons, Peres, and Schramm from \cite{blps}), who showed that such graphings admit measurable one-ended spanning trees with the same connected components as the original graphing.

Our results on measurable perfect matchings can also be applied to equidecompositions. Recall that in \cite{lacz11} Laczkovich solved the Tarski circle squaring problem and showed
that the unit-area disc and the unit square are equidecomposable
using translations in the plane. Recently, Grabowski,
M\'ath\'e and Pikhurko \cite{gmp} showed that the circle
squaring is possible with measurable pieces, and Marks and
Unger \cite{marks-unger} gave a Borel solution  to the Tarski circle squaring problem (see also \cite{MatheICM} for more on the recent developments in this field. The Borel solution of Marks and Unger \cite{marks-unger} was the first to use a rounding algorithm for Borel flows on Schreier graphs of $\mathbb{Z}^d$.

In Section \ref{sec:mcs}, we show that the measurable circle squaring of Grabowski, M\'ath\'e and Pikhurko can be deduced from our results. 
We give a couple of proofs, which are different from those of Grabowski, M\'ath\'e and Pikhurko \cite{gmp} and of Marks and Unger \cite{marks-unger}. In particular, we do not refer to the discrepancy estimates of Laczkovich. The first proof reduces the circle squaring to find a perfect matching in a regular one-ended hyperfinite graph and uses Laczkovich's results on uniformly spread sets. The second proof relies on the fact that Laczkovich provides two independent sets of translations that witness circle squaring. This proof also shows that 
if two measurable sets admit two equidecompositions by two independent sets of vectors, 
then they are measurably equidecomposable by the union of these two sets of vectors.

\subsection{Future work}
The characterization of (not necessarily hyperfinite) bipartite graphings with measurable perfect matchings seems beyond reach. 
Very recently, the second author \cite{gabor.new} gave an example of a $d$-regular, acyclic, measurably bipartite graphing that admits no measurable perfect matching.
Our paper characterizes the hyperfinite case. The other case that can be handled is the case of bipartite graphings with large expansion, including actions of groups with the Kazhdan Property (T).
This idea plays a crucial role in the papers of Margulis \cite{margulis}, Sullivan \cite{sullivan} and Drinfeld \cite{drinfeld} on the Banach-Ruziewicz problem. 
Lyons and Nazarov \cite{lyons.nazarov} also use expansion in order to prove that non-amenable groups a.s. admit factor of iid perfect matchings. In \cite{gmp1} Grabowski, M\'ath\'e and Pikhurko used expansion to prove that bounded, measurable sets of nonempty interior and equal measure in $\mathbb{R}^n \text{ }(n \geq 3)$ admit a measurable equidecomposition. 
The question of which pairs of compact sets of equal measure in $\mathbb{R}^3$ admit a measurable equidecomposition still seems to be a hard problem to understand; see the work of 
Cie\'sla and Grabowski
\cite{grabowski.ciesla}). 



The non-bipartite case is much more technical, similarly as for finite graphs. Cs\'oka and Lippner \cite{csoka.lippner} proved that every non-amenable Cayley graph a.s. admits a factor of iid perfect matching.
We expect that this holds for amenable one-ended Cayley graphs.

\begin{question}
Does every one-ended Cayley graph admit a factor of iid perfect matching a.s.?  
\end{question}




Cs\'{o}ka, Lippner{\color{blue},} and Pikhurko \cite{clp} showed that bipartite graphings of maximal degree $d$ admit a measurable edge coloring with $d+1$ colors. Greb\'{i}k and Pikhurko \cite{gp} showed that this estimate works in the non-bipartite case, matching the general optimal bound in the finite case. A natural strengthening of Corollary \ref{regular} would be given by a positive answer to the following question.

\begin{question}
Does every $d$-regular one-ended bipartite hyperfinite graphing admit a measurable edge coloring with $d$ colors?
\end{question}

Another interesting question is connected with the measure preserving assumption in our results. In the non-pmp case, Conley and Miller
\cite{conley.miller.pm} showed that hyperfinite acyclic Borel graphs of degree at least $2$ and with no injective rays of degree $2$ on even indices admit measurable perfect matchings. It would be interesting to know if the results of the current paper hold in the non-pmp setting as well. The main difficulty seems to lie in extending the results on the 
existence of a.e. one-ended spanning trees in hyperfinite graphings used in Section \ref{sec:tilings} to the non-pmp case.


\subsection{Organization}

The paper is organized as follows. In Section \ref{sec:notation} we collect the basics and notation, and in Section \ref{sec:rounding} we state our main result, Theorem \ref{fmatching}, in its full generality and discuss some of its applications. In Section \ref{sec:fract-perf-match} we introduce some of the basic tools of our construction, prove Theorem \ref{odd} and discuss the structure of the perfect fractional matching polytope. Section \ref{sec:tilings} introduces another important tool, namely connected toasts, and shows how they can be used in constructing useful families of cycles.  In Section \ref{sec:fract-perf-match} we prove a special case of our main result, which implies Corollary \ref{regular}. Section \ref{one.lined.section} is devoted to the analysis of a complementary special case. Section \ref{section.final} gets the two previous special cases together and proves Theorem \ref{fmatching} in its full generality. In Section \ref{sec:factor-iid-perfect} we apply our results to prove Theorem \ref{fiid},  and in Section \ref{sec:mcs} we show how our result implies the measurable circle squaring.
The Appendix consists of the proofs of Proposition \ref{equi} on equivalent characterizations of a.e. two-ended graphings and of Lemma \ref{locfin} on finding locally finite one-ended hyperfinite subgraphings.

\subsection*{Acknowledgements}
The authors would like to thank Alexander Kechris, L\'aszl\'o Lov\'asz, Russell Lyons, Oleg Pikhurko, Mikael de la Salle and other participants of Damien Gaboriau's groupe de travail \textit{Actions!} in Lyon, as well as \'Ad\'am Tim\'ar for many helpful comments and discussions. 

\section{Notation and basics}\label{sec:notation}

Given a real number $x$, we denote by $\{x\}$ the fractional part of $x$ and by $\lfloor x \rfloor$ the integral part of $x$.

A countable Borel equivalence relation $E$ on a standard probability space $(V,\nu)$ is \textit{probability measure preserving} (pmp) if every partial Borel bijection whose graph is contained in $E$ preserves the measure. A countable Borel equivalence relation $E$ on a standard Borel space is \textit{hyperfinite} if it is an increasing union of Borel equivalence relations with finite equivalence classes. 

A locally countable Borel graph $G$ on a standard probability space $(V,\nu)$ is \textit{probability measure preserving} if the equivalence relation induced by its connected components is pmp. Equivalently, $G$ is pmp if for some (equivalently, any) sequence of Borel involutions $T_n$ with $E(G)=\bigcup_{n\in\mathbb{N}}T_n$, each $T_n$ preserves the measure (see \cite{kechris.marks}).

A locally countable Borel graph $G$ on a standard Borel space $V$ is \textit{hyperfinite} if the equivalence relation induced by its connected components is hyperfinite. A graphing is \textit{hyperfinite} if it is hyperfinite a.e. If a locally countable Borel graph $G$ is defined on a standard probability space $(V,\nu)$, then $G$ is a.e. hyperfinite if and only if for every $\varepsilon>0$ there exists $k$ and a Borel set $V'\subseteq V$ with $\nu(V\setminus V')<\varepsilon$ such that all the components of the graph induced by $G$ on $V'$ have size at most $k$ (see \cite{elek}). Any Schreier graph of an amenable group action is a.e. hyperfinite.

In this paper, a \textit{graphing} is a Borel, locally countable probability measure preserving graph on a standard probability space $(V,\nu)$. 
We use the standard graph theoretic notation and refer to $V$ as to $V(G)$ and denote the set of edges by $E(G)$. 
For a set $W\subseteq V$ we write $\partial W$ for  the set of edges between a vertex in $W$ and a vertex not in $W$ and $E(W)$ for the set of edges between vertices of $W$. For a set $W\subseteq V$ we write $N(W)$ for the set of vertices adjacent to at least one vertex in $W$.
A graph is \textit{$d$-regular} if the degree of every vertex is equal to $d$ and a graph is \textit{regular} if it is $d$-regular for some $d\in\mathbb{N}$.
Given a graphing $G$ on $(V,\nu)$, there is a natural probability measure on the set of edges $E(G)$ (see e.g. \cite[Section 18]{kechris.miller}), which we denote by $\mu$ defined for $F\subseteq E(G)$ as $\mu(F)=\frac{1}{2}\int_X\mathrm{deg}_F(x)d\nu$, where $\mathrm{deg}_F(x)$ is the degree in the spanning subgraph induced by $F$. Equivalently, \cite[18.2]{lovasz} for $A,B\subseteq V(G)$ we have $\mu(A\times B)=\int_A \mathrm{deg}_B(x)d\nu(x)$, where $\mathrm{deg}_B(x)$ denotes the number of edges from $x$ to $B$. We use the standard measure theoretic terminology regarding $(E(G),\mu)$, e.g., we say that two measurable sets are \textit{essentially equal} if their symmetric difference is a null set.

Throughout this paper we will work primarily with perfect fractional matchings. Given a locally countable graph (or a graphing) $G$, function $f:V(G)\rightarrow \mathbb{N},$ and a function $c:E(G)\rightarrow \mathbb{N},$ we say that a symmetric function  $\tau:E(G)\rightarrow [0,\infty)$ is a \textit{perfect fractional $f$-matching bounded by $c$} if $\tau(x,y)\leq c(x,y)$ for each edge $(x,y)\in E(G),$ and $\sum_{y\in N(x)}\tau(x,y)= f(x)$ for every  $x\in V(G)$. 
A \textit{perfect $f$-matching bounded by $c$} is a perfect fractional $f$-matching bounded by $c$, which takes values in $\mathbb{N}$.

Given a group $\Gamma$ and a probability measure on $[0,1]$ 
we consider its Bernoulli shift, i.e., $[0,1]^{\Gamma}$ equipped with the product measure and a natural probability measure preserving action of the group $\Gamma$. A Cayley graph $G$ of $\Gamma$ admits a \textit{factor of iid perfect matching} if there is a $\Gamma$-invariant measure on the set of perfect matchings of $G$ which is a factor of the Bernoulli shift.
Factor of iid processes can be phrased in terms of measurable subsets of the Bernoulli shift. For example, a Cayley graph of $\Gamma$ admits a factor of iid
perfect matching if and only if there exists a measurable a.e. perfect matching of the corresponding Schreier graphing on the Bernoulli shift.

Given a bounded Polish metric space $(X,\rho)$ and two Borel probability measures $\kappa_1,\kappa_2$ on $X$, a \textit{coupling} is a Borel probability measure on $X\times X$ whose marginals are $\kappa_1$ and $\kappa_2$. The \textit{Wasserstein distance} (a.k.a. the \textit{Kantorovich--Rubinstein distance})  $W(\kappa_1,\kappa_2)$ is defined as the infimum of $\int_{X\times X}\rho(x,y) d\kappa(x,y)$ where $\kappa$ is a coupling of $\kappa_1$ and $\kappa_2$.
The space $\mathcal{P}(X)$ of Borel probability measures on $X$ equipped with the Wasserstein distance is also a Polish metric space \cite[Theorem 6.18]{villani}, and the Wasserstein distance induces the weak topology on $\mathcal{P}(X)$ \cite[Theorem 6.9]{villani}.

A probability measure preserving (pmp) action of a group $\Gamma$ on a standard Borel space $X$ with a Borel probability measure $\mu$ (invariant under the action) is usually denoted by $\Gamma\curvearrowright(X,\mu)$. We say that an action is \textit{totally ergodic} if every infinite subgroup of $\Gamma$ acts ergodically on $X$. Any Bernoulli shift of $\Gamma$ is totally ergodic.



Given an infinite graph $G$, two semi-infinite paths 
are \textit{end-equivalent} if for every finite set $F\subseteq V$ they can be connected by a path contained in $V\setminus F$. 
An \textit{end} of $G$ is an end-equivalence class of semi-infinite paths. 
In particular, for a finite $n$ an infinite graph 
has $\textit{$n$ ends}$ if for every finite set $F\subseteq V$, the induced graph on $V\setminus F$ has at most $n$  connected components containing a semi-infinite path, and there exists a finite set $F\subseteq V$ such that the induced graph on $V\setminus F$ has exactly $n$ components containing a semi-infinite path.
For locally finite graphs, this is a special case of the definition of the space of topological ends of a topological space \cite{frendenthal}. 
%
%
Any hyperfinite graphing has at most two ends in a.e. connected component (see \cite[Theorem 5.1]{adams} for the locally finite case and \cite[Lemma 3.23]{jkl} for the general case).
We say that a graphing $G$ is \textit{$n$-ended} if the connected component of every vertex of $G$ has $n$ ends.  

It is well-known that if the Cayley graph of a finitely generated group is two-ended, then the group admits $\mathbb{Z}$ as a normal subgroup of finite index, (see \cite[Theorem 5.12]{scott.wall} or \cite{tointon.yadin}). More generally, a connected vertex-transitive graph is two-ended if and only if it is quasi-isometric to $\mathbb{Z}$ (see \cite{trofimov,losert} or \cite[Corollary 3.13]{two-ended}).

A graphing has \textit{linear growth} if for every vertex $x$ there exists a $C$ such that $|B(x,r)| \leq Cr$ for every $r$. Note that this is equivalent to saying that for every orbit there exists a constant $C$ such that for every vertex $x$ there is a constant $D$ such that $|B(x,r)|\leq Cr+D$.

The following proposition collects the equivalent characterizations of a.e. two-ended graphings. It shows that a graphing is a.e. two-ended if and only if it has a.e. linear growth. In particular, hyperfinite graphings of a.e. superlinear growth are one-ended. 
This was essentially known and mostly proved. 
Benjamini and Hutchcroft \cite[Theorem 1.2, Remark 2.3]{benjamini.hutchcroft} proved one of the implications and mentioned another, 
and many ingredients of the proof  essentially appear in their work, though they are phrased somewhat differently.
In \cite{cgmtd}, Conley, Gaboriau, Marks and Tucker-Drob proved for every graphing that a.e. superlinear growth implies the existence of an a.e. one-ended subforest, though they only included the proof in case of superquadratic growth (see 
\cite[Theorem 2.6 and Remark 2.7]{cgmtd}). We give a proof of the proposition below in the Appendix.
The proof uses the existence of one-ended subforests (whose components may be smaller than those of a graphing), which exist in any graphing with only infinite components that is nowhere two-ended. For amenable one-ended unimodular random graphs, the latter was proved by Tim\'{a}r \cite{timar}, and the general statement was proved by Conley, Gaboriau, Marks and Tucker-Drob \cite[Theorem 2.1]{cgmtd}.


\begin{prop}\label{equi}
Let $G$ be a locally finite graphing. The following are equivalent.

\begin{itemize}
\item[(i)] $G$ has linear growth a.e.

\item[(ii)]$G$ is two-ended a.e.

\item[(iii)]  There exists a measurable partition $V(G) = \bigcup_{n=1}^{\infty} A_n$ a.e. such that for every $n>0$ the induced graph on $A_n$ consists of infinite components, and there exists a measurable family $\mathcal{C}_n$ of connected non-adjacent sets of size at most $n$ such that  $\bigcup\mathcal{C}_n\subseteq A_n$ and the induced graph on $A_n \setminus \bigcup\mathcal{C}_ n$ has only finite, connected components, all of these are adjacent to exactly two sets of $\mathcal{C}_n$, and every vertex of $\mathcal{C}_n$ is adjacent to exactly two components of $A_n \setminus \bigcup\mathcal{C}_n$.


\end{itemize}
\end{prop}

We will be working in the setting when the vertex set $V$ is endowed with a probability measure $\nu$. The graphing $G$ does not need to be of bounded degree, and in general $\mu$ may not be a probability measure. However, 
the following lemma helps to reduce some some problems regarding locally countable graphings to the setting of 
locally finite graphings. 
The lemma is a version of \cite[Theorem 2.1]{cgmtd} for graphings which are not necessarily locally finite. We will use it in the hyperfinite case only, but the statement does not need this assumption. The proof of the lemma below is included in the Appendix.

\begin{lemma}\label{locfin}
Let $G$ be a 
graphing which is nowhere zero-ended or two-ended. Then $G$ admits a hyperfinite one-ended spanning subgraphing $H$ such that 
 $\mu(H)<\infty$, in particular, $H$ is a.e. locally finite.
\end{lemma}




The following lemma is implicit in  \cite{wehrung}, \cite{laczkovich.dec} or \cite{ciesla.sabok} in case of amenable group actions. 

\begin{lemma}\label{lemat}
Let $G$ be a hyperfinite locally finite graphing, $H\subseteq G$ a spanning subgraphing, 
$f:V(G)\to\mathbb{N}$ and 
 $a,b:E(G)\to[0,\infty)$ be measurable such that $a(e)\leq b(e)$ holds for every edge $e$.
Assume that there exists a perfect fractional $f$-matching on $G$ such that for every edge $e \in E(H)$ its value is in $[a(e),b(e)]$. Then there exists a measurable perfect fractional $f$-matching on $G$ with values in $[a(e),b(e)]$ for $e\in E(H)$.
\end{lemma}


\begin{proof}
First, we can find a probability measure $\mu'$ on $E(G)$ which is equivalent to $\mu$ (i.e. $\mu\ll\mu'$ and $\mu'\ll\mu$) and such that $a,b\in L^2(E(G),\mu')$.
Using the hyperfiniteness of $G$, 
take a sequence of measurable partitions $\mathcal{C}_n$ of measurable subsets $V_n\subseteq V$ into finite subsets such that $\nu(V_n)\geq 1-\frac{1}{2^n}$. 
Construct a sequence of functions $g_n\in L^2(E(G),\mu')$ with $g_n(e)\in[a(e),b(e)]$ for all $e\in E(H)$. For each $n$ and an element $F$ of $\mathcal{C}_n$ choose the lexicographically least function $\varphi_n(F)$ on the edges in $F$ which can be extended to a perfect fractional $f$-matching bounded by $a$ and $b$ on the edges of $H$.
For each $n$, let $g_n\in L^2(E(G),\mu')$ be any function that extends the union of $\varphi_n(F)$ for all $F\in \mathcal{C}_n$ and such that $\|g_n\|_2\leq\|b\|_2$. Each function $g_n$ is a perfect $f$-matching on the vertices that are in the interior of a cell in $\mathcal{C}_n$. The sequence $g_n$ has a weakly convergent subsequence in $L^2(E(G),\mu')$ and write $g$ for its limit. By Mazur's lemma, there exists a sequence of convex combinations of $g_n$'s that converges to $g$ in $L^2(E(G),\mu')$ and that easily implies that $g$ is a.e. a fractional perfect $f$-matching that is bounded by $a$ and $b$ on the edges of $H$. 
\end{proof}

The following basic lemma will be used in Section \ref{sec:tilings}.

\begin{lemma}\label{subgraph}
Let $G$ be a finite, connected graph and $N \subseteq V(G)$ a subset of even size.
Then there exists a spanning subgraph $H$ of $G$ such that every vertex of $P$ has odd degree in $H$, and every vertex of  $V(G) \setminus P$ has even degree in $H$.
\end{lemma}

\begin{proof}
We prove by induction on the size of $N$. If $N=\emptyset$ then $H$ can be the edgeless spanning subgraph. Else choose two different vertices $s,t \in N$. 
Let $L$ denote a path connecting $s$ and $t$. By induction, there exists a spanning subgraph $H'$ of $G$ such that the degree of each vertex in $N \setminus \{ s,t \}$ 
is odd, and the degrees of the other vertices are even in $H'$. Let $E(H)$ be the symmetric difference of $E(L)$ and $E(H')$.
\end{proof}

\section{Measurable fractional matchings}
\label{sec:rounding}

In this section we collect some definitions and 
facts concerning measurable fractional matchings that will be used in the remainder of the paper and give a general formulation of our main results.  

Our main results will show that one can convert measurable perfect fractional matchings into measurable perfect matchings and it applies to hyperfinite one-ended graphings in 'most' circumstances (including regular graphings). However, the example below shows that we cannot in general round measurable fractional matchings in one-ended graphings without some additional assumptions.

\begin{example} \label{example}
There exists a hyperfinite one-ended bounded degree (measurably) bipartite graphing that admits a measurable perfect fractional matching but no measurable matching.
\end{example}

\begin{proof}
We begin with a construction of a bipartite graphing since it is slightly simpler than that of a measurably bipartite graphing. 

Recall that in \cite{lacz3} Laczkovich constructed a 2-regular graphing without a measurable perfect matching.  Namely, consider the irrational rotation on the circle $x\mapsto x+\alpha\ (\mathrm{mod}\ 1)$ (with $\alpha<\frac{1}{2}$) and let $G_\alpha$ be the Schreier graph of the induced action. The rotation by $2\alpha$ is ergodic
and this implies that there is no measurable perfect matching in $G_\alpha$ (cf. Claim \ref{folklore}).

Now, let $\beta$ be such that $\alpha,\beta$ and $1$ are linearly independent over the rationals, and consider the action of $\mathbb{Z}^2$ on $[0,1]$ induced by the rotations by $\alpha$ and $\beta$. Denote by $G_{\alpha,\beta}$ its Schreier graphing. For each edge coming from $G_\beta$ replace that edge by a cycle of length $4$, connected by two opposite edges, as in Figure \ref{fig:3.2}, and denote this graphing by $G$. In order to define the probability measure on $V(G)$ let $H$ be the five-vertex graph consisting of a vertex connected by an edge to a four-cycle. Consider the set $[0,1] \times V(H)$ with the uniform probability measure and the bijection $V(G) \rightarrow [0,1] \times V(H)$ mapping every vertex of $V(G)$ to a pair: the first coordinate is the closest vertex of $G_{\alpha,\beta}=[0,1]$ to the left, and the second coordinate is its position in the isomorphic copy of $H$ consisting a vertex of $G_{\alpha,\beta}$ and the four-cycle to its right. This induces a probability measure on $V(G)$. It is not difficult to see that $G$ is hyperfinite, given that $G_{\alpha,\beta}$ is hyperfinite. $G$ is pmp because $G_\alpha$ is pmp and the extra edges are covered by six pmp involutions.


\begin{figure}[ht]
    \centering
    \includegraphics[width=\textwidth]{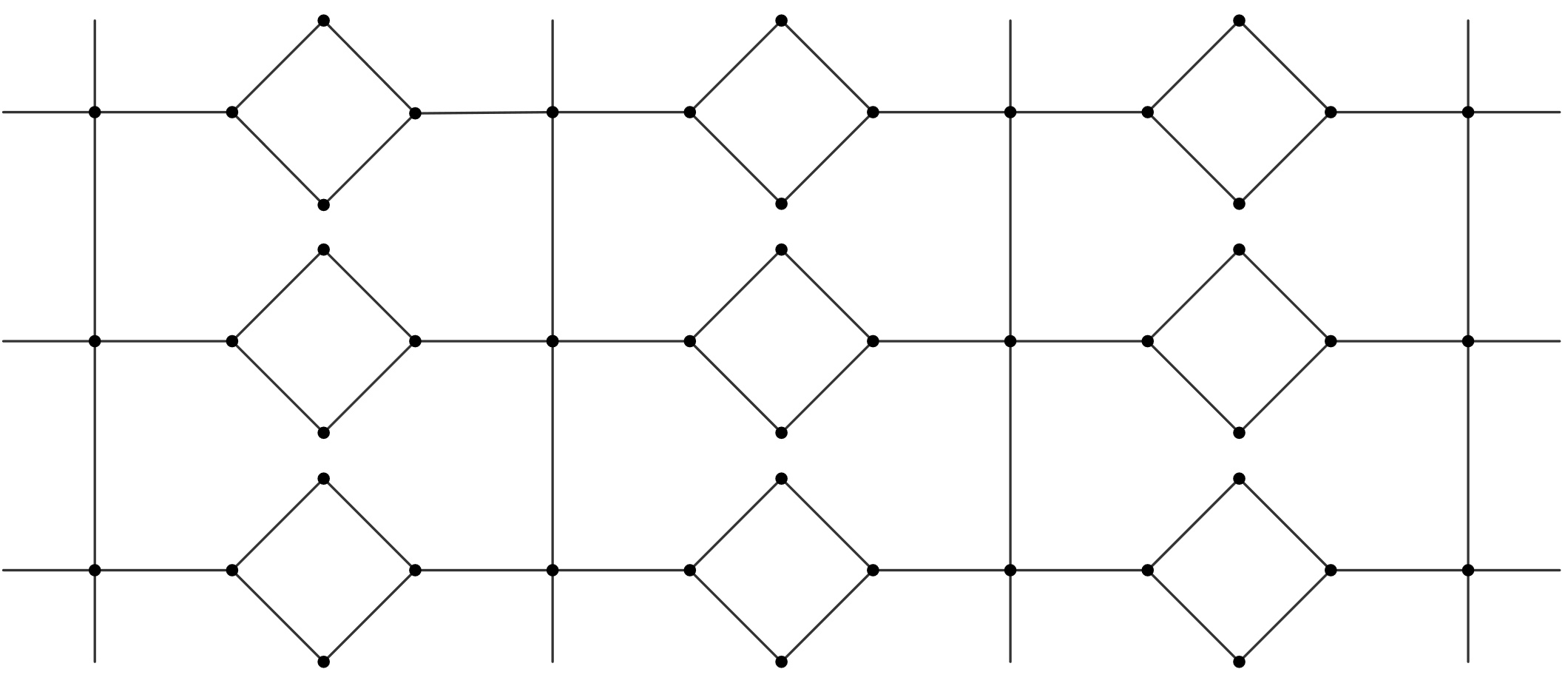}
    \caption{A component of the graph from Example \ref{example}}
    \label{fig:3.2}
\end{figure}

The fact that the graphing $G$ is hyperfinite  and one-ended follows from the hyperfiniteness and one-endedness of $G_{\alpha,\beta}$. $G$ admits a perfect fractional matching which is a perfect matching on the $4$-cycles and is equal to $\frac{1}{2}$ on the edges from $G_\alpha$. Note that a measurable perfect matching on $G$ would have to be a perfect matching on each $4$-cycle and hence its restriction to the edges of $G_\alpha$ would be a perfect matching in $G_\alpha$.

Now we describe the construction of a measurable bipartite graphing. It is based on the construction above and
the second author's construction from \cite{gabor}.
Let $0<\alpha, \beta <1$ be irrational and linearly independent over the rationals. Consider the following four isometries of the real line: $x \mapsto x, \text{ } x \mapsto x+ 2\alpha,\text{ } x \mapsto 2-x,\text{ } x \mapsto 2\alpha-x$,
and the measurably bipartite graphing $H_\alpha$ whose classes $I$ and $J$ are Borel isomorphic to the intervals $[0,1]$ and $[\alpha, 1+\alpha]$, respectively, and the set of edges corresponds to the union of the restrictions of the graphs of the four isometries. As shown in \cite{gabor}, the graphing $H_\alpha$ does not admit a measurable perfect matching. For every 
$i \in I, i \leq 1-\beta$ 
add a cycle of length $4$ to the graphing $H_\alpha$, and connect one vertex of the cycle to $i$ and the opposite vertex to $i+\beta$. Denote the induced graphing by $H_{\alpha,\beta}$. 

As before, the fact that $H_{\alpha,\beta}$ is hyperfinite follows from the fact that a Schreier graphing of an action of an amenable group is hyperfinite.
It is one-ended by Proposition \ref{equi}, since every component of $H_{\alpha,\beta}$ contains infinitely many components of $H_\alpha$. 

Again, $H_{\alpha,\beta}$ admits a measurable perfect fractional matching which is a perfect matching on the $4$-cycles and is equal to $\frac{1}{2}$ on the edges from $H_\alpha$. Any measurable perfect matching on $H_{\alpha,\beta}$ would have to be a perfect matching on each $4$-cycle and hence its restriction to the edges of $H_\alpha$ would be a perfect matching in $H_\alpha$.
\end{proof}

This example shows that it is possible to construct one-ended graphings and perfect fractional matchings that are forced to obtain specific values on some edges. After removing these forced edges, the remaining graphing may be two-ended and admit no measurable perfect matching. 
In order to avoid this problem we consider the following notion.

\begin{definition}
Given a graphing $G$, measurable $f:V(G)\rightarrow \mathbb{N}$ and $c:E(G)\to\mathbb{N}$ and a measurable perfect fractional $f$-matching $\tau$, the \textit{support} of $\tau$ with respect to $c$ is $$\supp(\tau,c)=\{e\in E(G): 0<\tau(e)<c(e)\}.$$ If $c$ is equal to $1$ everywhere, then we write $\supp(\tau)$ for $\supp(\tau,c)$.
\end{definition}

Now we can state our main theorem on rounding measurable perfect fractional matchings. 

\begin{theorem}\label{fmatching}
  Let $G$ be a bipartite graphing  $f: V(G) \to \mathbb{N}$ be integrable and $c: E(G) \to \mathbb{N}$ be measurable. Suppose that $G$ admits a measurable perfect fractional $f$-matching $\tau$ such that $\supp(\tau,c)$ is hyperfinite and nowhere two-ended .
Then $G$ admits a measurable perfect $f$-matchings $\sigma$ bounded by $c$ such that $|\sigma-\tau|<1$. Moreover, $\tau$ belongs to the closure of the convex hull of such measurable perfect $f$-matchings, with respect to the a.e. convergence. 


\end{theorem}
In particular, Theorem \ref{fmatching} implies that under its assumptions a.e. edge in $\supp(\tau,c)$ belongs to a measurable perfect $f$-matching bounded by $c$ such that $|\sigma-\tau|<1$.

The support is defined above for a given perfect fractional matching. However, there is always an essentially largest such support for a given graphing, as shown in the next lemma.

\begin{lemma}\label{maximal}
Let $G$ be a graphing, $f:V(G)\to\mathbb{N}$ and $c:E(G)\to\mathbb{N}$ be measurable. If there exists a measurable perfect fractional $f$-matching bounded by $c$ then there exists a measurable perfect fractional $f$-matching $\phi_{\max}$ bounded by $c$ such that for every measurable perfect fractional $f$-matching $\psi$
bounded by $c$ the set $\supp(\psi,c)\setminus \supp(\phi_{\max},c)$ is a nullset.
\end{lemma}

\begin{proof}
We define a sequence of perfect fractional $f$-matchings in the following way. Start with an arbitrary measurable perfect fractional $f$-matching $\phi_0$.
Assume that given the measurable perfect fractional $f$-matching $\phi_n$ there exists a measurable perfect fractional $f$-matching $\psi$ that does not satisfy the condition of the lemma (with $\phi_n$ replacing $\phi_{\max}$). 

Consider $\sup_{\psi} \mu(\supp(\psi,c) \setminus \supp(\phi_n,c))$, where the supremum is taken over all such perfect fractional $f$-matchings $\psi$,
and pick a $\psi_n$ for which $\mu(\supp(\psi_n,c) \setminus \supp(\phi_n,c))$ is at least the half of this supremum.
Set $\phi_{n+1}=  ( 1-\frac{1}{2^n}) \phi_n + \frac{1}{2^n} \psi_n$.
The sequence $\phi_n$ is uniformly convergent, and the limit $\psi_{\max}$ is a measurable perfect fractional $f$-matching bounded by $c$.
Note that $\supp(\psi_{\max},c) = \bigcup_{n=1}^{\infty} \supp(\phi_n,c)$.
Hence $\psi_{\max}$ satisfies the lemma by the choice of $\psi_n$.
\end{proof}

Finally, as a corollary to Theorem \ref{fmatching}, we can give our characterization of hyperfinite graphings with measurable perfect matchings --- via a reduction to the two-ended case. Note that a measurable perfect fractional matching of essentially maximal support can be found effectively as in the proof of 
Lemma \ref{maximal}.


\begin{cor}\label{characterization}
Let $G$ be a  bipartite hyperfinite graphing,  $f:V(G)\to\mathbb{N}$ be integrable and $c:E(G)\to\mathbb{N}$ be measurable. Let $\tau$ be a measurable perfect fractional $f$-matching $\tau$ with essentially maximal support $\supp(\tau,c)$.
The graphing $G$ admits a measurable perfect fractional $f$-matching a.e. if and only if the two-ended components of $\supp(\tau,c)$ admit a.e. a measurable perfect $f$-matching.
\end{cor}

\begin{remark}\label{constant1}
 For the rest of this paper we will work under the assumption that that $c$ is the constant function $1$, i.e., $|\tau(e)|\leq 1$ for every edge $e$. Proving Theorem \ref{fmatching} under this assumption is sufficient, since for arbitrary measurable $f,\tau,c$ we can consider $f',\tau'$ with $\tau'(e)=\{\tau(e)\}$ and $f'(v)=f(v)-\sum_{e: v\in e}\lfloor\tau(e)\rfloor$. We will also assume that $\supp(\tau)=G$ as we can always replace $G$ with $\supp(\tau)$. Finally, we may also assume that $f(v)\not=0$ for every vertex $v$ as we can work on the induced subgraph on $V'=\{v\in V: f(v)\not=0\}$.
\end{remark}

Let us also comment on the assumption that $f$ is integrable in the results above. The only place when we essentially use this assumption is Section \ref{one.lined.section} (Claim \ref{almostall}). In all other results we can work without this assumption, replacing the (possibly infinite) measure $\mu$ with an equivalent quasi-invariant finite measure $\mu'$ (as in the proof of Lemma \ref{lemat})

When applied to a constant function $f=1$, Theorem \ref{fmatching} gives the existence of measurable perfect matchings in a wide range of hyperfinite one-ended graphings, and Theorem~\ref{matching} from the introduction is a special case of Theorem~\ref{fmatching}. Corollary \ref{regular} follows when applied to the perfect fractional matching that is equal to $\frac{1}{d}$ everywhere, where $d$ is the degree of the regular graphing. We will see some applications of Corollary \ref{regular} in Sections \ref{sec:factor-iid-perfect} and \ref{sec:mcs}. Theorem~\ref{fmatching} can also be applied to find regular spanning subgraphings of arbitrary degree in one-ended regular graphings. This, in particular, implies Corollary~\ref{regular}.

\begin{cor}\label{subgraph toth}
Any $d$-regular hyperfinite one-ended bipartite graphing $G$ admits a $k$-regular spanning subgraphing for every $k<d$. \end{cor}

\begin{proof}
Consider the functions $\tau=\frac{k}{d}$ on every edge, $f=k$ on every vertex and apply Theorem~\ref{fmatching}.
\end{proof}

We can also use Theorem \ref{fmatching} to obtain a measurable \textit{balanced orientation} in a regular one-ended graphing of even degree, i.e., an orientation such that the in-degree equals the out-degree for every vertex (cf. \cite{bht}).
In \cite[Theorem 1.5]{thornton2022orienting} Thornton showed that any $d$-regular Borel graph admits a Borel orientation for which the out-degree of every vertex is in $[\frac{d}{2},\frac{d}{2}+1]$. In case of regular graphings of even degree $d$, Theorem \ref{fmatching} implies the existence of a measurable orientation, for which the out-degree of a.e. vertex is $\frac{d}{2}$, so that it is a balanced orientation.
\begin{cor}\label{balanced}
Consider a hyperfinite one-ended graphing $G$. Assume that every vertex has even degree. Then there exists a measurable balanced orientation of the edges. 
\end{cor}
\begin{proof}

Consider the barycentric subdivision $G'$ of $G$, that is, $V(G')=V(G) \cup E(G)$, and $E(G')=\{ (x,e): x \in V(G),\ e \in E(G),\ x \text{ is an endvertex of } e \}$. Note that $G'$ is a hyperfinite one-ended graphing. Consider the functions $f: V(G') \rightarrow \mathbb{N}$ defined as $f(e)=1$ for every $e \in E(G)$, $f(x)=\frac{1}{2}\mathrm{deg}_G(x)$ for every $x \in V(G)$. 
The constant half function on $E(G')$ is a fractional measurable perfect $f$-matching. Hence, by Theorem \ref{fmatching} there exists an  measurable perfect $f$-matching $\phi$ bounded by $1$. Consider the following orientation: an edge $(x,y)$ is oriented towards $x$ if $\phi ((x,(x,y)))=1$. This is a balanced orientation.
\end{proof}

In particular, this answers the question of Bencs, Hru\v{s}kov\'a and T\'oth \cite[Question 6.4]{bht}, as it shows that every unimodular vertex-transitive graph of even degree without a factor of iid balanced orientation is quasi-isometric to $\mathbb{Z}$. A factor of iid is defined for unimodular graphs (see \cite[Section 2.7]{bht}), 
and in the non-amenable transitive unimodular case, a factor of iid balanced orientation exists by \cite[Theorem 1]{bht1}, while in the amenable one-ended case it exists by Corollary \ref{balanced}. So this leaves only the two-ended graphs, which are quasi-isometric to $\mathbb{Z}$ by \cite[Corollary 3.13]{two-ended}.

\section{Rounding perfect fractional matchings and the matching polytope}
\label{sec:fract-perf-match}

If a finite bipartite graph  admits a perfect fractional matching, then it admits a perfect matching. A very elegant proof of this fact, due to Edmonds, is obtained by looking at the so-called \textit{perfect fractional matching polytope} of all perfect fractional matchings and noticing that an extreme point of this polytope must be a perfect matching. The latter follows from the fact that for an extreme point $\varphi$ of the matching polytope the set $\mathrm{supp}(\varphi)$ must be acyclic. In this section, we exploit this idea in the context of measurable perfect fractional matchings.

\begin{definition}
Given an even cycle $C$ in a graph with a distinguished edge and $\varepsilon \in \mathbb{R}$, an 
\textit{alternating $\varepsilon$-circuit} on $C$ is a function that is equal  to $\varepsilon$ on the even edges and $-\varepsilon$ on the odd edges (even and odd edges with respect to their distance from the distinguished one). 
\end{definition}
First, the simple idea of constructing a measurable perfect matching $\varphi$ for which $\supp(\varphi)$ is acyclic leads to a short proof of Theorem \ref{odd}.

In the proof we use the following standard notation: given a sequence $s\in\{-1,1\}^n$, we write $s^\frown1$ and $s^\frown-1$ for the concatenations of $s$ with $1$ and $-1$, respectively. For  $t\in\{-1,1\}^\mathbb{N}$ we write $t|n$ for $(t(0),\ldots,t(n-1))\in\{-1,1\}^n$
\begin{proof}[Proof of Theorem \ref{odd}]

We construct perfect fractional matchings $\tau_s$ for every $s\in\bigcup_{n=1}^\infty \{-1,1\}^n$ such that each $\tau_s$ takes values in $\{0,\frac{1}{d},\frac{2}{d},\ldots,1\}$ on every edge. We start with the empty sequence $s=
\emptyset$ and put $\tau_\emptyset$ to be the perfect fractional matching that takes the value $\frac{1}{d}$ on every edge. 

Suppose we have all $\tau_s$ for $s\in\{-1,1\}^n$.   For each such $s$, using 
\cite[Proposition 4.5]{kst} we can find a Borel set $\mathcal{C}_s$ of disjoint cycles (each with a distinguished edge) whose edges are contained in $\mathrm{supp}(\tau_s)$. 
For each $s\in\{-1,1\}^n$ we choose a Borel set $\mathcal{C}_s$ such that $\mu(\bigcup \mathcal{C}_s)$ is equal to at least half of the supremum of the measures of $\bigcup \mathcal{C}_s$ for all possible such $\mathcal{C}_s$.

Next, we define $\tau_{s^\frown1}$ and $\tau_{s^\frown-1}$.  To define $\tau_{s^\frown1}$ we add an alternating $\frac{1}{d}$-circuit 
to every cycle  in $\mathcal{C}_s$, and to define $\tau_{s^\frown-1}$ we subtract an alternating $\frac{1}{d}$-circuit 
to every cycle  in $\mathcal{C}_s$.

Note that 
\begin{equation}\label{norm}
\| \tau_{s^\frown1} \|^2_2 + \| \tau_{s^\frown-1} \|^2_2 = 2 \| \tau_s \|^2_2 + 
\frac{2}{d^2} \mu(\bigcup \mathcal{C}_s).
\end{equation}
Consider  the probability measure space on $\{-1,1\}^\mathbb{N}$ with the symmetric flip-coin measure and the random variables  $\|\tau_{t|n}\|_2$ and $\mu(\bigcup \mathcal{C}_{t|n})$. Then by (\ref{norm}) we have $$\mathbb{E}_t \| \tau_{t|n+1} \|^2_2 = \mathbb{E}_t \| \tau_{t|n} \|^2_2 +  \frac{1}{d^2} \mathbb{E}_t\mu(\bigcup \mathcal{C}_{t|n}).$$

Since each $\tau_s$ is bounded by $1$, the norms $\| \tau_s \|^2_2$ are bounded  by $\mu(E(G))=d$ and hence we get $$\mathbb{E}_t\sum_{n=1}^\infty \mu(\bigcup\mathcal{C}_{t|n})\leq d^3<\infty,$$ which means that for a.e. $t$ the sum $\sum_{n=1}^\infty \mu(\bigcup\mathcal{C}_{t|n})$ is finite. By the Borel--Cantelli lemma a.s. the value at a.e. edge will change only finitely many times in the sequence $\tau_{t|n}$. Write $\tau_t$ for the limit and note that it is a.e. a measurable perfect fractional matching.

A.s. the set $\supp(\tau_t)$ 
is essentially acyclic. Namely, we claim that $\supp(\tau_t)$ is acyclic whenever the sum $\sum_{n=1}^\infty \mu(\bigcup\mathcal{C}_{t|n})$ is finite. Indeed, otherwise we could find a family of pairwise disjoint cycles in $\supp(\tau_t)$ whose union would have a positive measure. But this would contradict that $\lim_{n\to\infty}\mu(\bigcup\mathcal{C}_{t|n})=0$, by our choice of $\mathcal{C}_{t|n}$ close to the optimum.

A.s. in a.e. component of $G$ the set $\supp(\tau_{t})$ is a tree and has no leaves. By hyperfiniteness, the average degree of vertices in $\supp(\tau_{t})$ is equal to $2$, so since it has no leaves, $\supp(\tau_{t_0})$ must be a disjoint union of lines in a.e. component. We know that $\tau_{t}$ has values which are rationals with denominator $d$, hence on each line these values alternate between $\frac{k}{d}$ and $\frac{d-k}{d}$ for some natural number $k$.  
Since $d$ is odd, this gives us a measurable choice of a perfect matching for a.e. such line and changing $\tau_{t}$ to such a measurable perfect matching on the lines gives a measurable perfect matching on the whole graphing $G$.

\end{proof}

The next several sections will be devoted to the proof of Theorem \ref{fmatching}. To this end, unless stated otherwise, we fix a hyperfinite bipartite one-ended graphing $G$ on $(V(G),\nu)$ and given 
an integrable $f:V(G)\to\mathbb{N}$ we consider the following set 



\begin{align*}
P_f=\{ \sigma\in L^2(E(G)): \sigma \text{ is a perfect fractional $f$-matching bounded by }1 \}
\end{align*}

Clearly, $P_f$ is a convex subset of $L^2(E(G))$. 
It is bounded, as for $\sigma\in P_f$, we have $\|\sigma\|_2^2\leq\|\sigma\|_1\|\sigma\|_\infty$ by the H\"older inequality and the fact that $\|\sigma\|_1=\frac{1}{2}\|f\|_1$. 
Note that $P_f$ is closed in the weak$^*$ topology of $L^2(E(G))$. Hence $P_f$ is compact (as a bounded subset of $L^2(E(G))$), separable and metrizable in the weak$^*$-topology. 


The Krein--Milman theorem implies that $P_f$ is a convex closure of the set of extreme points, in particular, it contains an extreme point. While for finite graphs $G$, any extreme point of $P_f$ must be a perfect matching, in the context of graphings the situation turns out to be a bit more subtle. In general, extreme points of $P_f$ do not need to be a.e. integral, as shown in Example \ref{example}. However, they have a nice structure, as we will see below.

 

\begin{lemma}\label{extreme}
Suppose that $G$ is a bipartite hyperfinite graphing, $f:V(G)\to\mathbb{N}$ is an integrable function 
and $\sigma \in P_f$ is an extreme point. Then
for a.e. $e \in E(G)$ we have $$\sigma(e)\in\{0,\frac{1}{2},1\},$$ and the set $\{e\in E(G):\sigma(e)=\frac{1}{2}\}$ 
is essentially a vertex-disjoint union of bi-infinite paths.
\end{lemma}

\begin{proof}
First we show that the set of edges, where $\sigma$ is not integral, is essentially acyclic. 
Suppose that the set of edges that belong to a cycle on which $\sigma$ is not integral has a positive measure. We can assume that the measure is finite and positive and pass to a locally finite  subgraphing spanned by cycles on which $\sigma$ is not integral such that the measure of the edges in this subgraphing is positive. Using a Borel coloring of the intersecting cycles \cite[Proposition 4.5]{kst}, we would find an $\varepsilon>0$ and a set of pairwise disjoint cycles of positive measure such that we could add or subtract alternating $\varepsilon$-circuits to $\sigma$ on these disjoint cycles and still be in $P_f$. However, this is not possible for an extreme point. Thus, the set of non-integral edges of $\sigma$ is essentially acyclic.

Note that there is no vertex covered by exactly one edge where $\sigma$ is not integral. Since the set of non-integral valued edges is acyclic and $G$ is hyperfinite, the subgraph spanned by these edges should have a.e. degree two or zero. Hence, the set of edges, where $\sigma$ is not integral is essentially a vertex-disjoint union of bi-infinite paths. 

On every such path $\sigma$ alternates between two values. Consider the set of paths where $\sigma$ is not half. If the set of such paths had a positive measure we could add or subtract for a possibly smaller set of paths of positive measure and an $\varepsilon > 0$ an alternating $\pm\varepsilon$-valued measurable function to $\sigma$ and still stay in $P_f$, which is not possible for an extreme point. This completes the proof of the lemma.
\end{proof}

Throughout this paper, given a graphing $G$ and an integrable  $f:{V(G)}\to\mathbb{N}$ 
we write $$R_f=\{\chi\in P_f: \chi(e)\in\{0,\frac{1}{2},1\} \mbox{ for a.e. } e\in E(G)\}.$$
Note that by Lemma \ref{extreme}, the extreme points of $P_f$ are in $R_f$. For $\chi \in R_f$ 
set $$L(\chi)=\{e\in E(G):\chi(e)=\frac{1}{2}\}.$$

Later, we will prove that given a measurable perfect fractional $f$-matching $\chi\in R_f$, we can find an extreme point $\chi'\in P_f$ such that $L(\chi')$ has smaller measure than $L(\chi)$ if the latter has positive measure. 


\section{Covering lines with cycles via connected toasts}
\label{sec:tilings}
 In this section we start by constructing special tilings of one-ended hyperfinite graphings. We then use them to construct families of cycles in the graphings which cover given families of lines.  
 
 The construction of the tilings uses recent results of Tim\'ar \cite{timar} and Conley, Gaboriau, Marks and Tucker-Drob \cite{cgmtd} on measurable one-ended trees, expanding on prior work of Benjamini, Lyons, Peres, and Schramm \cite{blps}. The following definition refines the notion of a toast structure (see \cite[Definition 2.9]{gjks}, \cite[Definition 4.1]{gjks.forcing}), coined by Miller and motivated by the work of Conley and Miller \cite{conley.miller.toast} (in particular, a \textit{toast structure} is any tiling that satisfies Properties (1) and (2) of the definition below).

\begin{definition}\label{toast}
  Given a Borel graph $G$, we say that a Borel collection $\mathcal{T}$ of finite connected subsets of $V(G)$ is a \textit{connected toast structure} if it 
  satisfies

\begin{enumerate}
    \item{$\bigcup_{K\in \mathcal{T}}E(K)=E(G)$,} 

    \item{for every pair $K, L \in \mathcal{T}$ either $(K\cup N(K)) \cap L = \emptyset$ or $K \cup N(K) \subseteq L$, or }$L \cup N(L) \subseteq K$,


    \item{for every $K \in \mathcal{T}$ the induced subgraph on $K \setminus \bigcup_{K \supsetneq L \in \mathcal{T}} L$ is connected.}
    
\end{enumerate}
For a graphing $G$, we say that it \textit{admits a connected toast structure a.e.} if there exists a Borel co-null set $V'\subseteq V$ such that there exists a connected toast structure for the graph $G|V'$.

\end{definition}

Note that if $G$ admits a connected toast structure then it must also admit a Borel one-ended (component-wise) spanning tree. In the other direction, we may use one-ended spanning trees to construct connected toast structure a.e. 

\begin{prop}\label{tiling}
Any locally finite, one-ended hyperfinite graphing $G$ admits a connected toast structure a.e. 
\end{prop}

\begin{proof}

Any hyperfinite one-ended graphing contains a.e. a measurable one-ended spanning tree with the same connected components as the graphing by \cite[Lemma 2.10]{cgmtd}\footnote{In case of random amenable one-ended unimodular graphs, this was proved in \cite[Corollary 2]{timar}}. Let $T$ be such a tree for $G$.
We say that a vertex $v$ is of \textit{height $n$ in $T$} if the maximum downward directed path in $T$ starting from $v$ is of length $n$. We call a finite set $L$ a \textit{tile of height $n$} if it consists of a vertex of height $n$ and the vertices below it in $T$.
Further, we say that $K$ is \textit{covered by $L$} if $K \cup N(K) \subseteq L$.

First, we construct an increasing sequence of integers $n_1<n_2<\ldots$ and a sequence $\mathcal{T}_1, \mathcal{T}_2,\ldots$ of families of tiles such that 

\begin{enumerate}
    \item[(i)] the height of every tile in $\mathcal{T}_i$ is at least $n_i$ and less than $n_{i+1}$,
    \item[(ii)] $\nu(\bigcup{\mathcal{T}_i})>1-2^{-i}$ for every $i$,
    \item[(iii)] for every $i$, the total measure of the tiles in $\mathcal{T}_i$ which are not covered by a tile in $\mathcal{T}_{i+1}$ is at most $2^{-i}$. 
\end{enumerate}

Towards this, observe that for any finite connected $C\subseteq V(G)$ there are infinitely many integers $m\in\mathbb{N}$ such that $C$ is covered by a tile of height $m$.
Having defined $n_0,\ldots,n_j$ and $\mathcal{T}_1,\ldots,\mathcal{T}_j$ satisfying the desired properties (i)---(iii), we can find an $n_{j+1}>n_j$ such that 
\begin{itemize}
    \item all but an arbitrary small proportion of the tiles in $\mathcal{T}_j$ are covered by tiles of height at least $n_j$ and less than $n_{j+1}$
     \item all but an arbitrary small proportion of the vertices are contained by tiles of height at least $n_j$ and less than $n_{j+1}$
\end{itemize}
Let $\mathcal{T}_{j+1}$ consist of the maximal tiles of height at least $n_j$ and less than $n_{j+1}$.

Now, we define $\mathcal{T}'$ as follows. A tile $K$ belongs to $\mathcal{T}'$  if there exists $i\in\mathbb{N}$ such that $K \in \mathcal{T}_i$ and $K$ is covered by a tile in $\mathcal{T}_{i+1}$

First, note that almost every vertex is contained in some tile of $\mathcal{T}'$. Property (ii) of the construction and the Borel--Cantelli lemma show that a.e. vertex is contained by a tile in $\mathcal{T}_i$ for all but finitely many $i$. Property (iii) of the construction and the Borel--Cantelli lemma imply that for a.e. vertex all but finitely many tiles containing it are in $\mathcal{T}'$.

For every edge $e$ there is a vertex $v$ such that the endvertices of $e$ are contained by a tile if and only if $v$ is in the tile. And for every vertex there are only finitely many such edges, hence (1) of the definition is satisfied for $\mathcal{T}'$. 

To obtain the connected toast strucutre, we will modify $\mathcal{T}'$ by gluing some tiles together, in order to ensure that property (2) is satisfied. More precisely, let $\mathcal{R}$ be the relation that relates tiles $K$ and $L$ if $K\cap L=\emptyset,$ but $(N(K)\cup K)\cap L\neq \emptyset,$ and let $\sim$ be the smallest equivalence relation on the set of tiles that contains  $\mathcal{R}$. We put  $\mathcal{T}=\{\bigcup[K]_\sim: K\in \mathcal{T}'\}$ and we claim that $\mathcal{T}$ is a connected toast structure.

The property (1) is satisfied for $\mathcal{T}$ since it was satisfied for $\mathcal{T}'$.
Note that any two tiles $K,L\in\mathcal{T}'$ are either disjoint or one contains the other. 
By construction, if a tile in $\mathcal{T}'$ contains another then it also covers it.  Moreover, if $L$ covers $K$ then also $L$ covers $\bigcup[K]_\sim$. This implies that each element of $\mathcal{T}$ is finite as well as the property (2) of the definition of a connected toast structure. 
The property (3) of the definition is satisfied because if $K\in\mathcal{T}$, then $K\setminus \bigcup_{K \supsetneq L \in \mathcal{T}} L= K\setminus \bigcup_{K \supsetneq L \in \mathcal{T'}} L$ is connected by the edges of the tree $T$.

\end{proof}



In the remaining part of this section, we construct families of cycles in a graphing that cover a family of lines in that graphing. In particular, given an extreme point $\chi\in P_f$, this will allow us to cover the lines of $L(\chi)$ with cycles. We will use this in the next sections to decrease the measure of $L(\chi)$ until it becomes zero.


\begin{lemma} \label{1end} 
Consider a hyperfinite one-ended graphing $G$ and a subgraphing $L$ on $(X,\mu)$. Assume that $L$ is the vertex-disjoint union of bi-infinite lines. 

Then for every $\varepsilon>0$ and $k \in \mathbb{N}$
there exist Borel families  $\mathcal{C}_1, \dots ,\mathcal{C}_k$, each consisting of pairwise edge-disjoint cycles such that 
\begin{itemize}
    \item every edge not in $L$ is covered by at most one cycle in $\bigcup_{i=1}^k\mathcal{C}_i$, 
    \item $\mu(\bigcap_{i=1}^k((\bigcup\mathcal{C}_i) \cap L))>\mu(L)(1-\varepsilon)$.
\end{itemize}
\end{lemma}


  

\begin{proof}
We may assume that $G$ is locally finite and even $\mu(G)<\infty$, since $G$ contains a locally finite hyperfinite a.e. one-ended spanning subgraphing $H$ by Lemma~\ref{locfin}, and we may consider the spanning subgraphing with edge set $E(H) \cup E(L)$.
Let $\mathcal{T}$ be a connected toast structure for the graphing $G$ given by Proposition \ref{tiling}. Let $\mathcal{M}_1 \subset \mathcal{T}$ denote the family of minimal sets (ordered by containment), $\mathcal{M}_2$ 
denote the family of minimal sets in $\mathcal{T} \setminus \mathcal{M}_1$ etc. Obviously, $\mathcal{T} = \bigcup_{k=1}^{\infty} \mathcal{M}_k$.
Choose $m$ large enough such that at least $\mu(L)(1-\varepsilon)$ of the edges of $L$ is contained by a set in $\mathcal{M}_m$. 
We construct for $i=1, \dots ,k$ a family of pairwise edge-disjoint cycles $\mathcal{C}_i$ contained in the sets of $\mathcal{M}_{m+i}$.

For every $T \in\mathcal{M}_{m+i-1}$ let $L_T$ denote the set of edges in $E(L)$ with at least one end-vertex in $T$. Note that $L_T$ is a.e. an edge-disjoint union of finite paths with end-vertices not contained by the sets of $\mathcal{M}_{m+i-1}$.

Since $T \setminus \bigcup\mathcal{M}_{m+i-1}$ is connected, by Lemma \ref{subgraph} applied to the set $N_T$ of endpoints of $L_T$, there exists a set of edges $H_T \subseteq \bigcup\mathcal{M}_{m+i}$ such that $L_T \cup H_T$ has vertices of even degree only, and
$H_T \subseteq E(T) \setminus \bigcup_{M \in \mathcal{M}_{M+i-1}} E(M)$. 
Note that $H_T \cup L_T$ is an edge-disjoint union of cycles and denote this collection of cycles by $\mathcal{C}_T$. 
For each $i\leq k$ we put $\mathcal{C}_i$ to be the union of the families $\mathcal{C}_T$ for $T$ in $\mathcal{M}_{m+i}$.
Every such family contains all edges of $L\cap \bigcup\mathcal{M}_m$, which has measure at least $\mu(L)(1-\varepsilon)$.
\end{proof}

While the above lemma will suffice to obtain the results in the next section (in particular in case of a regular graphing), the general case will require a somewhat more subtle analysis. We will use the following definition to analyse general graphings.

\begin{definition}
Assume that $L$ is a subgraphing of a graphing $G$. We say that $L$ is \textit{one-lined} if $L$ consists of a single connected bi-infinite line in every component. 
\end{definition}

The next two lemmas will be used to decrease the measure of $L(\chi)$ for an extreme point $\chi\in P_f$ in case $L(\chi)$ is not one-lined. They will be used in the proof of our main results in the general case.

\begin{lemma} \label{2end} 
Let $G$ be a locally finite hyperfinite two-ended graphing $G$ and $L$ a subgraphing. Assume that $L$ is the vertex-disjoint union of bi-infinite lines and $L$ is nowhere one-lined.

Then for every $\varepsilon>0$ and every $k \in \mathbb{N}$
there exist Borel families  $\mathcal{C}_1, \dots ,\mathcal{C}_k$ each consisting of pairwise edge-disjoint cycles such that 
\begin{itemize}
    \item every edge not in $L$ is covered by at most one cycle in $\bigcup_{i=1}^k\mathcal{C}_i$, 
    \item $\mu(\bigcap_{i=1}^k((\bigcup\mathcal{C}_i) \cap L))>\frac{2}{3}\mu(L)(1-\varepsilon)$.
\end{itemize}
\end{lemma}
\begin{proof}
Note that a.e. component of $G$ contain only finitely many lines, since  components with infinitely many lines have superlinear growth and thus are one-ended by Proposition \ref{equi}.  We can assume that all components have the same number, say $l$, lines (by working separately with the union of the components with the same number of lines). By our assumption $l>1$.

By Proposition \ref{equi}, find a family $\mathcal{S}$ of finite cutsets in $G$ such that cutsets are connected, each component of $G\setminus\bigcup\mathcal{S}$ is connected and is adjacent to exactly two cutsets in $\mathcal{S}$.
The family $\mathcal{S}$ is bi-infinite on a.e. line of $L$. 
Also, given $S,T\in\mathcal{S}$ we write $[S,T]$ for the set of vertices that lie on or between these two cutsets $S$ and $T$, according to this bi-infinite structure. We also refer to the sets of the form $[S,T]$ as to \textit{intervals}.

Given $S,T\in\mathcal{S}$ write $L(S,T)$ for the vertices that lie on an interval of a line in $L$ such that the interval intersects $\bigcup\mathcal{S}$ only on its endvertices, one in $S$ and one in $T$. By shrinking $\mathcal{S}$ if necessary we can assume that for every $S,T\in\mathcal{S}$ the set $L(S,T)$ is a disjoint union of exactly $l$ intervals of lines in $L$. 

We can find a Borel family $\mathcal{I}$ of disjoint intervals of the form $[S,T]$ for $S,T\in\mathcal{S}$ such that more than $\mu(L)(1-\varepsilon)$ 
of edges in $L$ lie on $L(S,T)$ for $S,T$ being two endpoints of the same interval in $\mathcal{I}$.

Now we fix one $[S,T]\in\mathcal{I}$ and work locally on its induced subgraph. First, if $l$ is odd, then choose one of the shortest of the $l$ intervals in $L(S,T)$ and remove it from $L(S,T)$. Write $L'(S,T)$ for $L(S,T)$ without this chosen interval. Note that $L'(S,T)$ is a union of an even number of intervals in lines in $L$.

Next, for $i\leq k$ find families $\mathcal{I}_i(S,T)$ of pairwise disjoint sets of the form $[S_i,T_i]\subseteq[S,T]$ such that taking the union of all $L(S_i,T_i)$ for $[S_i,T_i]\in \mathcal{I}_i(S,T)$ for all $[S,T]\in\mathcal{I}$ we still cover at least $\mu(L)(1-\varepsilon)$ of the edges in $L$. Moreover, since  $L'(S,T)$ is obtained by removing at most one (among the shortest) of $l$ intervals from $L(S,T)$, and only when $l\geq3$, we can ensure that the union of all $L'(S,T)\cap L(S_i,T_i)$ (for $[S_i,T_i]$ and $[S,T]$ as above) will cover at least $\frac{2}{3}\mu(L)(1-\varepsilon)$
edges in $L$.

The set $L(S_i,T_i)\cap L'(S,T)$ is a union of an even number of intervals of lines in $L$. Apply Lemma \ref{subgraph} to the endvertices of these paths in $S_i$ and $T_i$, respectively, in order to get edge-disjoint cycles covering these paths with all the additional edges in $S_i$ and $T_i$.

Write $\mathcal{C}_i$ for the family of all cycles obtained in the above way for all $L(S_i,T_i)$ for $[S_i,T_i]\in \mathcal{I}_i(S,T)$ for all $[S,T]\in\mathcal{I}$. Note that we cover all edges in $L$ that lie in $L(S_i,T_i)\cap L'(S,T)$, and the latter set has measure at least $\frac{2}{3}\mu(L)(1-\varepsilon)$. 
\end{proof}

\begin{lemma}\label{covering.twolined}
Let $G$ be a hyperfinite graphing $G$ and $L$ a subgraphing of $G$. Assume that $L$ is the vertex-disjoint union of bi-infinite lines and $L$ is nowhere one-lined. 
Let $f:V(G)\to\mathbb{N}$ be integrable and $\tau:E(G)\to[0,1]$ be a measurable perfect fractional $f$-matching.

There exist $\theta>0$ such that  for every $k \in \mathbb{N}$
there exist Borel families  $\mathcal{C}_1, \dots ,\mathcal{C}_k$, each consisting of pairwise edge-disjoint cycles such that 
\begin{itemize}
    \item every edge not in $L$ covered by at most one cycle of $\bigcup_{i=1}^k\mathcal{C}_i$ 
    \item $\mu(\bigcap_{i=1}^k((\bigcup\mathcal{C}_i) \cap L))>\frac{1}{2}\mu(L)$,
    \item $\tau(e),1-\tau(e)>\theta$ for every edge $e$ covered by $\bigcup_{i=1}^k\mathcal{C}_i$.
\end{itemize}
\end{lemma}
\begin{proof}
Fix $k\in\mathbb{N}$. Given $\theta>0$ consider the spanning subgraphing $H_\theta$ spanned by the set of edges $E(L) \cup \{ e\in E(G):  \tau(e),1-\tau(e)\geq \theta\}$. Note that $H_\theta$ is locally finite. Choosing $\theta$  small enough, we can make sure that the measure of the union of components of $H_\theta$ where $L$ is one-lined is arbitrarily small. 
We can apply Lemma \ref{1end} (with $\varepsilon=\frac{1}{2}$) to the union of the one-ended components of $H_\theta$ and Lemma \ref{2end} (with $\varepsilon=\frac{1}{4}$) to the union of two-ended components of $H_\theta$  in order to obtain the required family of cycles.
\end{proof}

\section{Random distortions of perfect fractional matchings}

\label{sec:improving}


Over the next couple of sections, we show how, given an extreme point $\chi$ of $P_f$ one can improve it in the way that the measure of $L(\chi)$ can be decreased, until it becomes $0$.

In this section, will first do it under the additional assumptions of Proposition \ref{successor1}, which will already suffice to obtain Corollary \ref{regular}, as stated in Theorem \ref{uniformly.bounded}. To prove it, we will study what happens when we randomly distort a perfect fractional matchings by adding to it small random numbers.
We will use the following terminology.

\begin{definition}
Suppose $G$ is a bipartite graphing, and $f: V(G) \to \mathbb{N}$ is integrable. 
By a \textit{random measurable perfect fractional $f$-matching} we mean a Borel probability measure on $P_f$.
\end{definition}

Given a random measurable perfect fractional $f$-matching $X$ on $R_f$, and a real-valued function $g$ on $R_f$, we write $\mathbb{E}g(X)$ for the integral with respect to $X$. Also, given a Borel probability measure $X$ on $P_f$ we will write $\mathrm{bar}(X)$ for the barycenter of $X$.


\begin{proposition}\label{successor1}
Let $G$ be a hyperfinite bipartite one-ended graphing, $f:V(G)\to\mathbb{N}$ be integrable 
and $\tau:E(G)\to[0,1]$ be a measurable perfect fractional $f$-matching such that $\supp(\tau)=E(G)$. 

Let $\chi$ be a measurable perfect fractional $f$-matching such that
$\chi \in R_f$ 
and $\mu(L(\chi))>0$. 

Suppose that at least one of the following holds 
\begin{itemize}
\item[(a)]  $\{e\in E(G):\tau(e),1-\tau(e)>\theta\}$ is one-ended for some $\theta>0$,
\item[(b)] or $L(\chi)$ is nowhere one-lined.
\end{itemize}

Then there exists a random measurable perfect fractional matching $X$ such that  $X$ is an extreme point of $P_f$ a.s. and
\begin{itemize}
    \item[(i)] $\mathbb{E} \mu(L(X))< \mu(L(\chi))$  
    \item[(ii)] $
\mathbb{E}\mu(\{e\in E(G):X(e)\not=\chi(e)\})\leq 3(\mu(L(\chi))-\mathbb{E}\mu(L(X)))$.
\end{itemize}
and $\mathrm{bar}(X)$ is a convex combination of $\tau$ and $\chi$.
\end{proposition}

Before we prove the proposition above, we state and prove the following lemma. Its statement is slightly stronger than what we need in this section (namely in Theorem \ref{uniformly.bounded} we will use it only in case $X$ is a measurable perfect $f$-matching) but we will use this stronger statement in the next sections.

\begin{lemma}\label{weird.estimate}
Let $G$ be a hyperfinite bipartite graphing and $f:V(G)\to\mathbb{N}$ be integrable.
Let $\chi$ be a measurable perfect fractional $f$-matching such that
$\chi \in R_f$ 
and $\mu(L(\chi))>0$.

Suppose $X$ a random measurable perfect fractional $f$-matching $X$ which is a.s. an extreme point of $P_f$ 
such that

$$    \mathbb{E} \int_{e\in E(G)\setminus L(\chi)} |X(e)-\chi(e)| < \frac{1}{2}\mathbb{E} 
    \int_{e\in L(\chi)} |X(e)-\chi(e)|
$$

Then the following hold.
\begin{itemize}
    \item[(i)] $\mathbb{E} \mu(L(X))< \mu(L(\chi))-\mathbb{E} [\frac{1}{2}
    \int_{e\in L(\chi)} |X(e)-\chi(e)|- \int_{e\in E(G)\setminus L(\chi)} |X(e)-\chi(e)|]$  
    \item[(ii)] $
\mathbb{E}\mu(\{e\in E(G):X(e)\not=\chi(e)\})\leq 3(\mu(L(\chi))-\mathbb{E}\mu(L(X)))$.
\end{itemize}
\end{lemma}

\begin{proof}
Consider the sets (treated as random variables)
$$A=\{e\in E(G)\setminus L(\chi): X(e) \neq \chi(e)\},\quad B=\{e\in L(\chi): X(e)\neq \chi(e)\}.$$
and note that $B=L(\chi)\setminus L(X)
.$
(Technically, we will only treat $\mu(A)$ and $\mu(B)$ as random variables.) Note that by Lemma \ref{extreme} we have
$$\mathbb{E} \mu(A)\leq
2\mathbb{E} \int_{e\in (G)\setminus L(\chi)} |X(e)-\chi(e)|$$ and $$\mathbb{E}\mu(B)=2\mathbb{E} \int_{e\in L(\chi)} |X(e)-\frac{1}{2}|=2\mathbb{E} \int_{e\in L(\chi)} |X(e)-\chi(e)|.$$
Write $\gamma=\mathbb{E} \big[\frac{1}{2}
    \int_{e\in L(\chi)} |X(e)-\chi(e)|- \int_{e\in E(G)\setminus L(\chi)} |X(e)-\chi(e)|\big]$, so that we have

\begin{equation}\label{change1}
    \mathbb{E}\mu(A)<\frac{1}{2}\mathbb{E}\mu(B)-\gamma.
\end{equation}
Since 
$L(X) \setminus L(\chi) \subseteq A$ and $B=L(\chi) \setminus L(X)$, we get
\begin{equation}\label{change2}
\mathbb{E}\mu(L(X))\leq \mu(L(\chi))-\mathbb{E}\mu(B)+\mathbb{E}\mu(A),
\end{equation}
so by (\ref{change1}) we get $$\mathbb{E} \mu(L(X))\leq\mu(L(\chi))-\mathbb{E}\mu(B)+\frac{1}{2}\mathbb{E}\mu(B)-\gamma<\mu(L(\chi))-\gamma,$$
which justifies (a).
To justify (b), note that (\ref{change2}) together with (\ref{change1}) imply
\begin{eqnarray*}
  \mathbb{E}\mu(\{e\in E(G):X(e)\not=\chi(e)\})=\mathbb{E}\mu(A)+\mathbb{E}\mu(B)\\ \leq  3(\mu(L(\chi))-\mathbb{E}\mu(L(X))).  
\end{eqnarray*}
\end{proof}

Now we prove the main result of this section.

\begin{proof}[Proof of Proposition \ref{successor1}]
By Lemma \ref{weird.estimate}, it is enough to find a measurable perfect fractional matching $X$ which is a.s. an extreme point of $P_f$ 
such that

$$    \mathbb{E} \int_{e\in E(G)\setminus L(\chi)} |X(e)-\chi(e)| < \frac{1}{2}\mathbb{E} 
    \int_{e\in L(\chi)} |X(e)-\chi(e)|$$
and $\mathrm{bar}(X)$ is a convex combination of $\chi$ and $\tau$.

First we find a measurable perfect fractional matching $Y$ that satisfies the inequality, but may not be supported on extreme points.

There exists $\theta>0$ such that for every $k\in\mathbb{N}$ there exist Borel families of cycles $\mathcal{C}_1,\ldots,\mathcal{C}_k$, each consisting of pairwise edge-disjoint cycles and such that 

\begin{itemize}
    \item every edge not in $L$ is covered by at most one cycle of $\bigcup_{i=1}^k\mathcal{C}_i$, 
    \item $\mu(\bigcap_{i=1}^k(E(\bigcup\mathcal{C}_i) \cap E(L)))>\frac{1}{2}\mu(L)$,
    \item $\tau(e),1-\tau(e)>\theta$ for every edge $e$ covered by $\bigcup_{i=1}^k\mathcal{C}_i$.
\end{itemize}
Indeed, if (a) is satisfied, then apply Lemma~\ref{1end} to to $L=L(\chi)$, $\varepsilon=\frac{1}{2}$ and the subgraphing spanned by $\{e\in E(G):\tau(e),1-\tau(e)>\theta\}$ for the $\theta$ given in (a). If (b) is satisfied, then apply Lemma~\ref{covering.twolined} to $L=L(\chi)$ and get $\theta>0$ such that the above are satisfied.


We choose  $k$ large enough to be specified later and let $\mathcal{C}_1,\ldots,\mathcal{C}_k$ be Borel families of cycles as above. Put 
\begin{equation}\label{epsilon1}
\varepsilon=\frac{1}{2k+\frac{4}{\theta}}\quad\mbox{and}\quad\lambda=\frac{2\varepsilon}{\theta}
\end{equation}
and consider the perfect fractional $f$-matching $$\rho=\lambda \tau + (1-\lambda) \chi.$$
Note that for every edge $e$ covered by a cycle in $\mathcal{C}$ we have 
\begin{equation}\label{lowerbound}
 \lambda \tau(e) + (1-\lambda) \chi(e) ,\  1- \lambda \tau(e) + (1-\lambda) \chi(e) >\lambda\theta
\end{equation}

For every cycle in $\bigcup_{i=1}^k\mathcal{C}_i$ select one edge in a Borel way. For every $i\leq k$ we consider the alternating $\varepsilon$-circuit $\zeta_i$ on the cycles in $\mathcal{C}_i$, i.e., $\zeta_i(e)$ is equal to 0 on edges $e$ that do not belong to such cycles and for every cycle in $\mathcal{C}_i$ the function $\zeta_i$ is an alternating $\varepsilon$-circuit on this cycle.

Let the random variables $Z_1,Z_2,\ldots Z_k\in\{-1,1\}$ be independent identically distributed with $\mathbb{E}(Z_i)=0$.
Write $Y$ for the random perfect fractional $f$-matching obtained by randomly adding the circuits on the cycles, i.e., $$Y=\rho+\sum_{i=1}^k Z_i \zeta_i.$$

Note that (\ref{epsilon1}) imply that $\varepsilon<\lambda\theta$ and thus by (\ref{lowerbound}) we have that $0\leq Y(e)\leq1$ for $e\in E(G)\setminus L$. For $e\in L$ we have $\frac{1}{2}-\lambda\leq\rho(e)\leq\frac{1}{2}+\lambda$ and thus $\frac{1}{2}-\lambda-k\varepsilon\leq Y(e)\leq\frac{1}{2}+\lambda+k\varepsilon$, so by (\ref{epsilon1}) we also get $0\leq Y(e)\leq1$. That means  that $Y$ is still in $P_f$, for every choice of $Z_1,Z_2,\ldots Z_k\in\{-1,1\}$  


Note that for every edge $e\in E(G)\setminus L(\chi)$ the value $Y(e)-\chi(e)$ has the same sign, depending only on whether $\chi(e)=0$ or $\chi(e)=1$, and thus $\mathbb{E}Z_i=0$ implies that

\begin{equation}\label{estimateonN1}
    \mathbb{E}|Y(e)-\chi(e)|=|\rho(e)-\chi(e)|\leq\lambda(\tau(e)-\chi(e)) 
\end{equation}


For every measurable fractional perfect $f$-matching $\sigma$ we have $\|\sigma\|_1=\frac{1}{2}\|f\|_1$, so, by (\ref{estimateonN1}) and Fubini's theorem, in the expected value we have
\begin{equation}\label{estimateonN}
    \mathbb{E}\int_{e\in E(G)\setminus L(\chi)} |Y(e)-\chi(e)|
\leq \lambda(\|\tau\|_1+\|\chi\|_1)=\frac{2\varepsilon}{\theta}(\|\tau\|_1+\|\chi\|_1)=\frac{2\varepsilon}{\theta}\|f\|_1.
\end{equation}

Write $L'=L(\chi)\cap\bigcap_{i=1}^k\mathcal{C}_i$ and note that $\mu(L')\geq\frac{1}{2}\mu(L)$ by the properties of our families of cycles.
Recall that the Berry--Esseen theorem implies that if $X_1,\ldots,X_n\in\{-1,1\}$ are iid with $\mathbb{E}X_i=0$, then $\lim_{n\to\infty}\mathbb{E}|\sum_{i=1}^n X_i|\slash\sqrt{n}=\mathbb{E}|N|>0$, where $N$ has normal distribution, and thus $\mathbb{E}|\sum_{i=1}^n Y_i|
=\Omega(\sqrt{n})$. Therefore, for every edge $e\in L'$ we have
\begin{equation}\label{clt}
\mathbb{E}|Y(e)-\frac{1}{2}|=\varepsilon\,\Omega(\sqrt{k}).
\end{equation}
and the exact value of $\Omega(\sqrt{k})$ in (\ref{clt}) does not depend on the edge $e\in L(\chi)$, so by Fubini's theorem we have 
\begin{equation}\label{clt13}
 \mathbb{E}\int_{e\in L(\chi)}|Y(e)-\frac{1}{2}|\geq\varepsilon\,\Omega(\sqrt{k})\mu(L')\geq \varepsilon\,\frac{1}{2}\Omega(\sqrt{k})\mu(L(\chi)).
\end{equation}

Now, using (\ref{clt13}) and (\ref{estimateonN}) we can fix $k$ big enough so that 

\begin{equation}\label{final1}
    \mathbb{E}\int_{e\in M} |Y(e)-\chi(e)| < \frac{1}{2}
    \mathbb{E}\int_{e\in L} |Y(e)-\chi(e)|.
\end{equation}
Note that the barycenter of $Y$ is $\rho=\lambda\tau+(1-\lambda)\chi$, since $\mathbb{E} Z_i=0$ for each $i\leq k$.

Now, we improve $Y$ to be concentrated on the extreme points of $P_f$. By the Choquet–-Bishop--de Leeuw theorem, there exists a probability measure $X$ on the set of extreme points of $P_f$ whose barycenter is $\mathrm{bar}(Y)$.

 


Note that for a.e. $e \notin L(\chi)$ either $Y(e)-\chi(e)$ is always non-negative or $Y(e)-\chi(e)$ is always non-positive. Hence its absolute value is linear and  
$$\mathbb{E} \int_{e\in E(G)\setminus L(\chi)}  |Y(e)-\chi(e)| = \mathbb{E}\int_{e\in E(G)\setminus L(\chi)} |X(e)-\chi(e)|.$$
On the other hand,  for a.e. $e\in L(\chi)$ we have $\chi(e)=\frac{1}{2}$, so by convexity of $|x-\frac{1}{2}|$ on $[0,1]$, Jensen's inequality and Fubini's theorem imply $$\frac{1}{2}\mathbb{E}
    \int_{e\in L(\chi)} |Y(e)-\chi(e)| \leq
\frac{1}{2}\mathbb{E}
    \int_{e\in L(\chi)} |X(e)-\chi(e)|.$$

\end{proof}

Using the previous proposition, we can state and prove the following weaker version of Theorem \ref{fmatching}, which implies Corollary \ref{regular}. The proof of Theorem \ref{fmatching} in Section \ref{section.final} will basically follows the same lines but will need a more detailed analysis of Section \ref{one.lined.section} (this will be needed to prove Theorem \ref{fmatching} in its full generality, which is useful e.g. for applications such a those in \cite{timar.new}).

\begin{theorem}\label{uniformly.bounded}
Let $G$ be a bipartite hyperfinite  graphing $G$ and $f: V(G) \to \mathbb{N}$ be integrable. Suppose that $G$ admits a measurable perfect fractional $f$-matching $\tau:E(G)\to[0,1]$ such that for some $\theta>0$ the set 
$$\{e\in E(G):\tau(e),1-\tau(e)>\theta\}\ \ \mbox{is one-ended.}$$
Then $G$ admits such a measurable perfect $f$-matching $\chi:E(G)\to[0,1]$.
\end{theorem}
Theorem \ref{uniformly.bounded} can be strenghtened to say that $\tau$ is equal to a barycenter of a probability measure on $P_f$ concentrated on the set of measurable perfect $f$-matchings, see Remark \ref{barycenter.regular} after the proof of Theorem \ref{fmatching}.
\begin{proof}
We will find a perfect $f$-matching $\chi$ in $P_f$. 


By the Krein--Milman theorem, there exists $\chi_0\in P_f$ which is an extreme point of $P_f$. Starting with $\chi_0$, by induction, we define a sequence of measurable perfect fractional $f$-matchings $\chi_\alpha\in R_f$ for countable ordinals $\alpha$, so that
\begin{itemize}
    \item[(i)] $\mu(L(\chi_{\alpha+1}))<\mu(L(\chi_\alpha))$ if $\mu(L(\chi_\alpha))>0$,
    \item[(ii)] $\mu(\{e\in E(G):\chi_{\alpha+1}(e)\not=\chi_\alpha(e)\})<3(\mu(L(\chi_\alpha))-\mu(L(\chi_{\alpha+1})))$,
    \item[(iii)] $\chi_\gamma=\lim_{\alpha<\gamma}\chi_\gamma$ for limit $\gamma<\omega_1$.
\end{itemize}


Suppose we have constructed $\chi_\alpha$. By  Lemma \ref{successor1} there exists a random perfect fractional $f$-matching $X$ concentrated on $\mathrm{ext}(P_f)$ such that
$$    \mathbb{E} \int_{e\in E(G)\setminus L(\chi_\alpha)} |X(e)-\chi_\alpha(e)| < \frac{1}{2}\mathbb{E} 
    \int_{e\in L(\chi_\alpha)} |X(e)-\chi_\alpha(e)|
$$

Since $X$ is concentrated on $\mathrm{ext}(P_f)$, there exists an extreme point $\chi_{\alpha+1}$ of $P_f$ for which
$$   \int_{e\in E(G)\setminus L(\chi_\alpha)} |\chi_{\alpha+1}(e)-\chi_\alpha(e)| < \frac{1}{2} 
    \int_{e\in L(\chi_\alpha)} |\chi_{\alpha+1}(e)-\chi_\alpha(e)|
$$

Using  Lemma \ref{weird.estimate} (treating $\chi_{n+1}$ as a Dirac random perfect fractional matching) we get (i) and (ii) as needed.

At limit stages $\gamma$, note that
the condition (i) implies
the sequence $\mu(L(\chi_\alpha))$ for $\alpha<\gamma$ is strictly decreasing and (ii) implies that $\chi_\alpha$ is a.e. convergent by the Borel--Cantelli lemma. We put $\chi_\gamma\in R_f$ to be the limit of this sequence.

Note that there must exist $\alpha<\omega_1$, such that have $\mu(L(\alpha))=0$ and then this $\chi_\alpha$ is a measurable perfect $f$-matching a.e. as needed.
\end{proof}

\section{The one-lined case}
\label{one.lined.section}

For extreme points $\chi$ of $P_f$ which are one-lined, we will need a slightly different argument in order to decrease $\mu(L(\chi))$. The goal of this section is to prove the next proposition, which is a version of Proposition \ref{successor1} in the one-lined case.

\begin{proposition}\label{successor.oneline}
Let $G$ be a bipartite one-ended graphing and $f:V(G)\to\mathbb{N}$ be integrable.
Suppose that $\chi$ is a measurable perfect fractional $f$-matching such that
$\chi \in R_f$  
and $\mu(L(\chi))>0$. Assume that $\chi$ is one-lined on a set of components of positive measure. Let $\delta>0$.

There exists a random measurable perfect fractional matching $X$ which is a.s. in $R_f$ such that

\begin{itemize}
    \item[(i)] $\mathbb{E} \mu(L(X))< \mu(L(\chi))$  
    \item[(ii)] $\mathbb{E}\mu(\{e\in E(G):X(e)\not=\chi(e)\})\leq 3(\mu(L(\chi))-\mathbb{E}\mu(L(X)))$.
\end{itemize}
and $\|\mathrm{bar}(X)-\chi\|_2\leq\delta(\mu(L(\chi))-\mathbb{E} \mu(L(X)))$.
\end{proposition}

The key to prove Proposition \ref{successor.oneline} is to find a family of cycles we can use for the rounding. First, we need to introduce some notation. 

\begin{definition} Let $G$ be a bipartite graphing and $f:V(G)\to\mathbb{N}$ be integrable. Suppose $\chi$ is an extreme point of $P_f$ such that $L(\chi)$ is one-lined everywhere.
\begin{itemize}
    \item[(i)]  Write $E_{\chi=1}=\{e\in E(G):\chi(e)=1\}$ and $E_{\chi=0}=\{e\in E(G):\chi(e)=0\}$. 
    
    
    \item[(ii)] A {\it $\chi$-augmenting path} is a path of odd length $x_0, \dots, x_{2j-1}$ such that we have $(x_{2l},x_{2l+1})\in E_{\chi=0}$ for $0 \leq l \leq j$ and  $(x_{2l-1},x_{2l})\in E_{\chi=1}$ for $1 \leq l \leq j-1$. 
    
    \item[(iii)] A \textit{$\chi$-augmenting cycle} is a cycle  obtained by joining a $\chi$-augmenting path  with endvertices on $L(\chi)$ and the finite subinterval of $L(\chi)$ joining the two endpoints of the $\chi$-augmenting path.
    
    \item[(iv)] For a $\chi$-augmenting cycle $C\subseteq E(G)$ we write  $C_{\chi=1}$ for the set of edges of $C$ that contains  every other edge of $C$ extending $E_{\chi=1}\cap C$, and analogously we define $C_{\chi=0}$.
    
    \item[(v)] $\chi$-augmenting cycles in the same component of $G$ can be naturally divided into two equivalence classes induced by the proper $2$-coloring of the edges of the infinite line (this is done in the single component). 
    Two $\chi$-augmenting cycles $A$ and $B$ are \textit{equivalent} if every $e\in A_{\chi=1}$ has the same color as every edge in $B_{\chi=1}$ (and the same holds for $A_{\chi=0}$ and $B_{\chi=0}$).
\end{itemize}
\end{definition}

In particular, the definition of  a $\chi$-augmenting path requires that all vertices along the path are distinct. Note, however, that since the graphing $G$ is bipartite, any walk satisfying the conditions of (ii) above contains a $\chi$-augmenting path obtained by removing the (even) cycles between repetitions of vertices. Note also that there are exactly two equivalence classes of $\chi$-augmenting cycles.

The following lemma gives us the family of cycles that will be useful in proving Proposition \ref{successor.oneline}.

\begin{lemma}\label{main.one.lined}
Let $G$ be a bipartite one-ended graphing and $f:V(G)\to\mathbb{N}$ be integrable. Let $\chi\in R_f$. Suppose that $L(\chi)$ is one-lined everywhere and there is no subset of $L(\chi)$ of positive measure that admits a measurable perfect matching.

For every $k\in\mathbb{N}$ and $\varepsilon>0$ there exists measurable collections $\mathcal{C}_1,\ldots,\mathcal{C}_k$ each consisting of $\chi$-augmenting cycles edge-disjoint in both equivalent classes such that
\begin{itemize}
\item every edge not in $L(\chi)$ is covered by at most four $\chi$-augmenting cycles in $\bigcup_{i=1}^k\mathcal{C}_i$,
\item for each $i\leq k$, every edge in $L$ is covered at most twice by $\chi$-augmenting cycles in $\mathcal{C}_i$,
\item $\mu(\bigcap_{i=1}^k\{e\in L(\chi):\mbox{e is covered 
by } \chi\mbox{-augmenting cycles}$ in  both equivalence classes in $\mathcal{C}_i\}) >\mu(L(\chi))(1-\varepsilon)$.
\end{itemize}
\end{lemma}

Before we prove Lemma \ref{main.one.lined}, let us see how it implies Proposition \ref{successor.oneline}.

\begin{proof}[Proof of of Proposition \ref{successor.oneline}]


By Lemma \ref{weird.estimate}, it is enough to find a random measurable perfect fractional matching $X$ which is a.s. in $R_f$ such that

\begin{equation}\label{here}
    \mathbb{E} \int_{e\in E(G)\setminus L(\chi)} |X(e)-\chi(e)|<\frac{1}{2}\mathbb{E}
    \int_{e\in L(\chi)} 
     |X(e)-\chi(e)|
\end{equation}
and $\|\mathrm{bar}(X)-\chi\|_2\leq\delta(\frac{1}{2}\int_{e\in L(\chi)} 
     |X(e)-\chi(e)|
     - \mathbb{E} \int_{e\in E(G)\setminus L(\chi)} |X(e)-\chi(e)|)$.
     
First, suppose that there is a subset of $K\subseteq L(\chi)$ of positive measure such there is a measurable perfect matching on  $K$. 
Let $Z\in\{-1,1\}$ be uniformly distributed and define $X$ as follows. For edges $e\in K$ let $X(e)=\chi(e)+\frac{1}{2}Z$ if $e$ belongs to the perfect matching, and $X(e)=\chi(e)-\frac{1}{2}Z$ if $e$ does not belong to the perfect matching, while for $e\notin K$ let $X(e)=\chi(e)$. Note that $\mathrm{bar}(X)=\chi$ and the left-hand side in (\ref{here}) is zero, while the right-hand side in (\ref{here}) is positive, so $X$ is as needed.


Thus, for the rest of the proof we can assume that there is no measurable perfect $f$-matching on any subset of $L(\chi)$ of positive measure

Write $G_1$ for the union of the components of $G$ where $L(\chi)$ consists of a single line. Let $L=L(\chi)\cap G_1$. By our assumption, $\mu(L)>0$. 

Let $k\in\mathbb{N}$ be big enough and $\varepsilon>0$ be small enough. Apply Lemma \ref{main.one.lined} to get families $\mathcal{C}_1,\ldots,\mathcal{C}_k$ of $\chi$-augmenting cycles such that every edge is covered by at most four $\chi$-augmenting cycles in $\bigcup_{i=1}^k\mathcal{C}_i$ and for every $i\leq k$ every edge in $L(\chi)$ is covered by at most two $\chi$-augmenting cycles from $\mathcal{C}_i$, and 
all but $\varepsilon$ edges from $L(\chi)$ are covered by exactly one $\chi$-augmenting cycle from each equivalence class in $\mathcal{C}_i$.

We may assume that the cycles in $\bigcup_{i=1}^k\mathcal{C}_i$ have bounded length. Using \cite[Proposition 4.6]{kst}, we can find  $l\in\mathbb{N}$ and partition $\bigcup_{i=1}^k\mathcal{C}_i$ into $l$ edge-disjoint Borel subfamilies $\mathcal{D}_1\ldots,\mathcal{D}_l$. 

For each $j\leq l$ we define the function $\zeta_j$ as follows. $\zeta_j:E(G)\to\{-\frac{1}{2},0,\frac{1}{2}\}$ and $\zeta_j(e)$ is equal to $0$ on the edges $e$ that do not belong to $\bigcup\mathcal{D}_j$ and for every $\chi$-augmenting cycle $C\in\mathcal{D}_j$ the function $\zeta_j$ is a $\frac{1}{2}$-circuit on $C$, where the chosen edge is $C$ belongs to $C_{\chi=0}$.

Let $I\in\{1,\ldots,l\}$ be a random variable with the uniform distribution and let $Z\in\{-1,1\}$ be an independent random variable with the uniform distribution. Consider the following random perfect fractional matching $$X=\chi+ Z \zeta_I.$$

Note that 
\begin{equation}\label{outside}
\mathbb{E}\mu(\bigcup\mathcal{C}_I\setminus L(\chi))\leq\frac{4}{l} ,
\end{equation}
as a.e. edge in $E(G)\setminus L(\chi)$ is covered at most four times by $\bigcup_{i=1}^k\mathcal{C}_i$. Thus, by (\ref{outside})
\begin{equation}\label{upperbound}
\mathbb{E}\int_{e\in E(G)\setminus L(\chi)} |X(e)-\chi(e)| \leq \frac{2}{l},    
\end{equation}
 since  for every edge $e$ we have $|X(e)-\chi(e)|\leq\frac{1}{2}$.


Note also that for every edge $e\in L(\chi)$ covered by $\mathcal{C}_I$
we have 
\begin{equation}\label{lowerbound0}
|X(e)-\frac{1}{2}|=\frac{1}{2}    
\end{equation}
 with probability $\frac{1}{2}$.
Thus, by Fubini's theorem we get
\begin{equation}\label{lowerbound1}
\mathbb{E}\int_{e\in L(\chi)}|X(e)-\frac{1}{2}|\geq\frac{k}{4l}\mu(L(\chi))(1-\varepsilon),
\end{equation}
which follows from (\ref{lowerbound0}) and the fact that $\mathbb{E}\mu(\bigcup\mathcal{C}_I\cap L(\chi))\geq\frac{k}{l}\mu(L(\chi))(1- \varepsilon)$. The latter is true because every family $\mathcal{C}_1,\ldots,\mathcal{C}_k$ covers $L(\chi)$ up to $\varepsilon$. 

To see that (\ref{here}) is satisfied, note that by (\ref{upperbound}) and (\ref{lowerbound1}) we get
\begin{equation}\label{here2}
    \frac{1}{2}\mathbb{E}
    \int_{L(\chi)} 
     |X(e)-\chi(e)|-\mathbb{E} \int_{E(G)\setminus L(\chi)} |X(e)-\chi(e)|\geq\frac{1}{l}(\frac{k}{4}\mu(L(\chi))(1-\varepsilon)-2)
\end{equation}
which is positive as long as $k$ is big enough and $\varepsilon$ is small enough.

Now let us look at the barycenter of $X$. Write $L'=\{e\in L(\chi): e$ is covered by one cycle  $C\in\mathcal{C}$  as a $C_{\chi=1}$ edge and by one cycle $C\in\mathcal{C}$ as a
        $C_{\chi=0}\mbox{ edge} \}$. By our assumption, we have $\mu(L')\geq\mu(L(\chi))(1-\varepsilon)$. 
Note that the barycenter can change only at the edges in edges in $L(\chi)\setminus L'$ and at the edges in $E(G)\setminus L(\chi)$ covered by $\bigcup\mathcal{C}_I$. Each of the former edges is covered at most $2k$ times and the measure of such edges is at most $\varepsilon$. Each of the latter edges is covered at most four times and their measure is at most $\frac{1}{l}$. The absolute value of the change is $\frac{1}{2}$ with probability $\frac{1}{2}$. Thus, 
we get 
$\|\mathrm{bar}(X)-\chi\|_2 \leq  \frac{1}{4}(\frac{2}{l}+k\varepsilon)$. 
Thus, by (\ref{here2}) we are done if
 $\frac{1}{4}(\frac{2}{l}+k\varepsilon) \leq\frac{\delta}{l}\big(\frac{k}{4}\mu(L(\chi))(1-\varepsilon)-2\big)
$,
which holds if $k$ is big enough and $\varepsilon$ is small enough, relative to $k$. 
        





\end{proof}





The rest of this section is devoted to the proof of Lemma \ref{main.one.lined}. 

\begin{proof}[Proof of Lemma \ref{main.one.lined}]

We will need a series of claims.

\begin{claim}\label{almostall}
 A.e. vertex in $V(G)$ is available via a $\chi$-augmenting path from $L(\chi)$ and has an $E_{\chi=1}$-neighbor available via a $\chi$-augmenting path from $L(\chi)$.
\end{claim}
\begin{proof}
Write $A$ for the set of vertices available via an augmenting path from $L(\chi)$ and $B$ for the set of vertices in $A$ not adjacent to a vertex in $A$ via an edge in $E_{\chi=1}$. Also, denote by $D$ the set of neighbors of $B$ via an edge in $E_{\chi=1}$. We will show that $A$ is co-null and both $B$ and $D$ are null. This will be enough since every vertex has an $E_{\chi=1}$-neighbor (recall that $f>0$).

We have $B\subseteq A$ and $D\cap A=\emptyset$ and we will show that both $B$ and $D$ have measure zero. By the definition of $D$, for every edge $e$ from $B$ to the complement of $D$ we have $\chi(e)=0$. 
Note that every $E_{\chi=0}$-edge from $D$ goes to $A$, since there is a $\chi$-augmenting path to its other endvertex. Moreover, every $E_{\chi=0}$-edge from $D$ goes to $B$, else there would be a $\chi$-augmenting path to its endvertex in $D$. The latter implies that for every edge $e$ from $D$ to the complement of $B$ we have $\chi(e)=1$. In other words, $\chi(e)=1$ for all $e\in E(D)\cup\partial D\setminus \partial B$ and $\chi(e)=0$ for all $e\in E(B)\cup\partial B\setminus \partial D$.

Recall that for every set $W\subseteq V(G)$ and every measurable perfect fractional $f$-matching $\sigma$ we have $\int_W fd\nu=\int_{\partial(W)} \sigma d\mu+2\int_{E(W)}\sigma d\mu$. 
Then we can write 
$\int_{D} fd\nu-\int_{B} fd\nu$ as 

\begin{eqnarray*}
 \int_{\partial D\setminus \partial B}\chi d\mu + 2\int_{E(D)} \chi d\mu - \int_{\partial B\setminus \partial D}\chi d\mu - 2\int_{E(B)}\chi d\mu
\\ 
= \int_{\partial D\setminus \partial B}\tau d\mu + 2\int_{E(D)}\tau d\mu - \int_{\partial B \setminus \partial D}\tau d\mu - 2\int_{E(B)} \tau d\mu
\end{eqnarray*}
Since the first expression is maximal among the values where $\chi$ is replaced by another perfect fractional $f$-matchings, we must have $\tau(e)=1$ for a.e. $e\in E(D)\cup\partial D\setminus \partial B$ and $\tau(e)=0$ for a.e. $e\in E(B)\cup\partial B\setminus \partial D$. However, by our assumption $\tau(e)\in\{0,1\}$ happens on a null set of edges $e$. Thus, $\mu(\partial(B\cup D))=0$. The set $B\cup D$ is disjoint from $L(\chi)$, which implies that $B\cup D$ is a nullset. 

 Now, since $B$ is a nullset essentially every edge leaving $A$ is not in $E_{\chi=0}$, or
else the other endvertex would be available via a $\chi$-augmenting path from $L(\chi)$. In other words $\chi(e)=1$ for every edge $e\in \partial(A)$.
Let $F=N(A)\setminus A$ and note that every $E_{\chi=0}$ edge from $F$ must go to $A$, since there is a $\chi$-augmenting path to its other endvertex. Thus $\chi(e)=1$ for every edge $e\in \partial F\cup E(F)$. As previously,
we can write $\int_{F} fd\nu$ as $$\int_{\partial F} \chi d\mu+2\int_{E(F)}\chi d\mu=\int_{\partial F} \tau d\mu+2\int_{E(F)}\tau d\mu.$$
Again, since $\chi(e)=1$ for all $e\in\partial F\cup E(F)$, the left hand side is maximal possible among the values where $\chi$ is replaced by another perfect fractional $f$-matching bounded by $1$, we must have $\tau(e)=1$ for a.e. $e\in\partial F\cup E(F)$.  Hence $F$ is a nullset. Therefore a.e. vertex belongs to $A$.
\end{proof}

\begin{claim}\label{almostall2}
For a.e. component $C$ of $G$, for every finite subset $S$ of $C$ the following is true
\begin{enumerate}
    \item\label{uno} All but finitely many vertices of $C$ are available from $L(\chi)$ via a $\chi$-augmenting path contained in $C\setminus S$ and have an $E_{\chi=1}$-neighbor available from $L(\chi)$ via a $\chi$-augmenting path contained in $C\setminus S$.
  \item\label{due} For every finite interval $J\subseteq L(\chi)\cap C$ there exists a $\chi$-augmenting path connecting the two semi-infinite paths of $L(\chi)\setminus J$ in $C \setminus S$.
\end{enumerate}
\end{claim}

\begin{proof}

By Claim \ref{almostall} we can assume that every vertex in $V(G)$ is available via an augmenting path from $L(\chi)$ and has an $E_{\chi=1}$-neighbor available via such a path.

\medskip

(\ref{uno}) For every positive integer $k$ consider $f_k: V(G) \setminus L(\chi) \rightarrow V(G) \setminus L(\chi)$ mapping $x$ to the $k$-th vertex of the lexicographically first (with respect to any Borel coloring of directed edges with natural numbers) shortest augmenting path from $L(\chi)$ to $x$. Write $F$ for the graphing where we put an edge between $f_k(x)$ and $f_{k+1}(x)$ for every $x\in V(G)$. Note that $F$ is acyclic because $G$ is bipartite and thus the $F$-component of every point in $L(\chi)$ is a tree. 
Let $K\subseteq L(\chi)$ denote the set of points in $L(\chi)$ whose $F$-component is infinite. We claim that $K$ has measure zero. Towards a contradiction, suppose that $K$ has positive measure. 
 For each positive integer $n$ let $g_n$ be a Borel function on $K$ such that $g_n(x)$ is an element of the $F$-component of $x$, for $x\in K$, and $g_n(x)\not=g_m(x)$ for $n\not=m$. Considering the disjoint sets $g_n(I)$ we get a contradiction by the assumption that the graphing is pmp, as the sets $g_n(I)$ must have equal measure.

Thus, for a.e. point in $L(\chi)$ its $F$-component is finite. Given any finite set $S\subseteq C$, the union $S'$ of $F$-components intersecting $S$ is then finite, as well as the set $S''$ consisting of all $E_{\chi=1}$-neighbors of vertices in $S'$. Since every vertex has at least one $E_{\chi=1}$-neighbor (recall that $f>0$), any vertex which is not in $S''$ satisifes (\ref{uno}). 

\medskip

(\ref{due}) By (\ref{uno}), all but finitely many vertices are available from $L(\chi)$ via a $\chi$-augmenting path contained in $C\setminus S$ and have an $E_{\chi=1}$-neighbor available from $L(\chi)$ by such path.  Let $A_1$ denote the set of such vertices available from one of the semi-infinite paths of $L(\chi)\setminus J$ and $A_2$ the set of vertices available from the other semi-infinite path. 
Every vertex $v\in A_1\cup A_2$ has an $E_{\chi=1}$-neighbor $v'$ which is in $A_1\cup A_2$. If $v$ and $v'$  available from different infinite-subpaths, then we can join their $\chi$-augmenting paths via $(v,v')$ and obtain a $\chi$-augmenting path disjoint from $S$ joining the two semi-infintie paths of $L(\chi)\setminus J$. So we can assume that for every vertex $v\in A_1\cup A_2$ each its $E_{\chi=1}$-neighbor $v'$ which is in $A_1\cup A_2$ is available from the same semi-infinite path of $L(\chi)\setminus J$.

If $A_1\cap A_2\not=\emptyset$, then let $u\in A_1\cap A_2\not=\emptyset$. Choose a $E_{\chi=1}$-neighbor $u'$ of $u$ such that $u'\in A_1\cup A_2$. We can then find two $\chi$-augmenting paths disjoint from $S$ joining $u$ and $u'$ with distinct semi-infinite paths of $L(\chi)\setminus J$ and joining them along $(u,u')$ gives us the desired path.

Thus, suppose that $A_1$ and $A_2$ are disjoint.   Put $S'=C\setminus(A_1\cup A_2) $. $S'$ is finite and
since $C$ is one-ended and $A_1,A_2$ each contain a semi-infinite path, we know that there are infinitely many edges $e=(u,v)$ with $u\in A_1$ $v\in A_2$, and in particular there is such an $e$ with neither of its endvertices in $N(S')$. If $e\in E_{\chi=1}$ then joining the augmenting paths to $u$ and $v$ via $e$ gives the desired augmenting path joining the two semi-infinite paths.  If $e \in E_{\chi=0}$ then, by our assumption, we can find $E_{\chi=1}$-neighbor $u'\in A_1$ of $u$ and $v'\in A_2$ of $v$. Joining the $\chi$-augmenting paths to $u'$ and $v'$ with the three edges $(u,u')$, $(u,v)$ and $(v,v')$ gives us the desired path.

\end{proof}




\begin{claim}\label{coveringlines}
For almost every component $C$ of $G$, for every finite interval $J\subseteq L(\chi)\cap C$ and every finite subset $S \subseteq C \setminus V(L)$ there are $\chi$-augmenting cycles  in both equivalence classes covering $J$ and disjoint from $S$.
\end{claim}

\begin{proof}
Let $B$ be the set of those edges $e$ in $L$ such that it is not the case that for every finite interval $J\subseteq L(\chi)$ containing $e$ and every finite subset $S$ of $E(G)$ there are $\chi$-augmenting cycles in both equivalence classes covering $J$ and disjoint from $S$.
     
Note that for each line in $L(\chi)$ either the line is contained in $B$ or disjoint from $B$. Thus, we need to show that the set of lines contained by $B$ is a nullset.
Part (2) of Claim \ref{almostall2} implies that in a.e. component for every finite subset disjoint from the line, every finite subinterval of the line can be covered by a $\chi$-augmenting cycle in one of the equivalence classes disjoint from this set. If a line in $L(\chi)$ in such a component is contained in $B$, then for every edge $e$ in that line there is a  finite interval $J\subseteq L(\chi)$ containing $e$ and a finite subset $S$ of that component such that 
exactly one of the followings holds:
  \begin{itemize}
      \item[(a)] either for every  $\chi$-augmenting cycle $A$ covering $J$ and disjoint from $S$ we have $e\in A_{\chi=1}$
      \item[(b)] or for every  $\chi$-augmenting cycle $A$ covering $J$ and disjoint from $S$ we have $e\in A_{\chi=0}$.
  \end{itemize}
Note that this partitions a.e. edge of $B$ into type (a) or (b), and gives a measurable perfect matching on $B$, contradicting our assumption.
  \end{proof}

Given two $\chi$-augmenting cycles that share the same edge, we will want to obtain a $\chi$-augmenting cycle that uses that edge and parts of the two cycles. In order to define this correctly, we need the following definition. 

\begin{definition}
Given two equivalent $\chi$-augmenting cycles, we say that an edge not in $L(\chi)$ belongs to their \textit{substantial intersection}, if considering the orientations of the subpaths of the cycles on the line towards the same end, the extension of these orientations outside $L(\chi)$ agrees on this edge.

\end{definition}
Note that in the above definition it does not matter which orientation of the line $L(\chi)$ we choose. 

Crucially, note that if two equivalent cycles have an edge in their substantial intersection, then they can be merged into one $\chi$-augmenting cycle in (possibly) four ways as follows. Let $A$ and $B$ be two equivalent $\chi$-augmenting cycles and $e$ be an edge in their substantial intersection. Orient the subpaths of the cycles $A$ and $B$ on the line towards the same end and extend this orientation outside $L(\chi)$  to the union of the  $\chi$-augmenting paths. Using this orientation, write $A_1$ for the subpath of $A\setminus L(\chi)$  starting at the vertex on the line preceding $e$ in the orientation and ending at a vertex of $e$. Similarly define $A_2$ as the subpath of $A\setminus L(\chi)$ starting at a vertex of the edge $e$ following $e$ in the orientation and ending at a vertex of the line. Define $B_1$ and $B_2$ similarly for $B$. To \textit{merge} $A$ and $B$ we define the $\chi$-augmenting path that starts with either $A_1$ or $B_1$, is followed by $e$ and then by either $A_2$ or $B_2$. Then we close this $\chi$-augmenting path with the interval on $L(\chi)$ between its endvertices.



\begin{claim}\label{local.refinement}
 Let $\mathcal{C}$ be a finite collection of equivalent $\chi$-augmenting cycles that do not have substantial intersection outside $L(\chi)$ and $A$ be a $\chi$-augmenting cycle. Let $J\subseteq L(\chi)$ be a finite interval such that $J\subseteq A\cap L(\chi)$ and $J\cap\bigcup\mathcal{C}=\emptyset$. Then there exists a $\chi$-augmenting cycle $A'$ obtained by merging $A$ with at most two cycles $C$ and $D$ of $\mathcal{C}$ such that $J \subseteq A' \cap L(\chi)$, and the cycle $A'$ does not have substantial intersections with cycles in $\mathcal{C}\setminus\{C,D\}$. 

  
\end{claim}
\begin{proof}
 
  Note that $J$ divides $L(\chi)$ into two infinite components, so we refer to these two components as to the \textit{two sides} of $L(\chi)\setminus J$. 
For simplicity, suppose that there are $\chi$-augmenting cycles in $\mathcal{C}$ with vertices on  both sides of $L(\chi)\setminus J$ which are equivalent to $A$ and substantially intersect $A$ on $E(G)\setminus L(\chi)$ 
(if such cycles have vertices on $L(\chi)$ only on one side of $L(\chi)\setminus J$, then the proof is analogous to the proof below). 
 
Choose $\chi$-augmenting cycles $B$ and $C$ in $\mathcal{C}$ with vertices on $L(\chi)$ on different sides of $L(\chi)\setminus J$ 
such that $A$ is equivalent to $B$ and $C$, substantially intersects $B$ at an edge $e_1$ and $C$ at an edge $e_2$ and the edges between $e_1$ and $e_2$ in $A\setminus L(\chi)$ do not belong to the substantial intersection of $A$ and any cycle in $\mathcal{C}$ equivalent to $A$.  We merge $B$, $C$ and $A$ as follows.  Outside of the line $L(\chi)$ we first take the edges of $B\setminus L(\chi)$, followed by the edges of $A$ between $e_1$ and $e_2$ and then followed by the edges of $C$ outside of $L(\chi)$. This is a $\chi$-augmenting path joining two vertices of $L(\chi)$ and in order to get a $\chi$-augmenting cycle we add the vertices of the interval of $L(\chi)$ between these vertices.

Note that $A'$ covers $J$  since $B$ and $C$ are disjoint from $J$ and on different sides of $J$. 
Note that the cycle $A'$ does not intersect substantially any equivalent cycle in $\mathcal{C}\setminus\{B,C\}$ on $E(G)\setminus L(\chi)$ since cycles from $\mathcal{C}$ do not have substantial intersections between equivalent cycles outside of the line.
\end{proof}

Now, we prove the statement of Lemma \ref{main.one.lined}. More precisely, by induction on $k$ we prove that for every $\varepsilon>0$ there exist Borel families of $\chi$-augmenting cycles $\mathcal{C}^k_1, \dots ,\mathcal{C}^k_k$ such that

\begin{itemize}
\item there are no edges in $E(G)\setminus L(\chi)$ that belong to a substantial intersection of equivalent cycles in $\bigcup_{i=1}^k\mathcal{C}^k_i$,
\item for every $i\leq k$ there are no edges on $L(\chi)$ that belong to an intersection of equivalent cycles in $\mathcal{C}^k_i$
\item $\mu(\bigcap_{i=1}^k\{e\in L(\chi):\mbox{e is covered 
by } \chi\mbox{-augmenting cycles}$ in  both equivalence classes in $\mathcal{C}^k_i\}) >\mu(L(\chi))-\varepsilon$.
\end{itemize}

Note that this will imply the statement of Lemma \ref{main.one.lined} since there are two equivalence classes of $\chi$-augmenting cycles and if an edge belongs to an intersection of three equivalent $\chi$-augmenting cycles, then it belongs to a substantial intersection of two of them.

Assume that the statement holds for $k-1$, and let the families $\mathcal{C}^{k-1}_1, \dots ,\mathcal{C}^{k-1}_{k-1}$ be given by the induction assumption for $\frac{\varepsilon}{2}$.
We will construct the families $\mathcal{C}^k_1, \dots ,\mathcal{C}^k_k$ that work for $\varepsilon$. We choose a large $n$ and for each $j\leq n$ we construct families $\mathcal{C}_{1,j}, \dots ,\mathcal{C}_{k-1,j}$ as well as a family $\mathcal{D}_j$ such that

\begin{itemize}
    \item[(i)] $\mathcal{C}_{i,j+1}\subseteq\mathcal{C}_{i,j}$ and $\mu(\bigcup(\mathcal{C}_{i,j}\setminus \mathcal{C}_{i,j+1}))<\frac{\varepsilon}{4nk}$ for each $i<k$ and $j\leq n$,
    
\item[(ii)] each equivalence class of $\chi$-augmenting cycles in $\mathcal{D}_j$ is disjoint on $L(\chi)$ and covers at least $1-(\frac{1}{2})^j$ edges in $L(\chi)$
\item[(iii)] for each $j<n$  the family $\mathcal{D}_j\cup\bigcup_{i=1}^{k-1}\mathcal{C}_{i,j}$ does not have substantial intersections of equivalent cycles outside of $L(\chi)$. 
\end{itemize}

Choose $n$ large enough such that $(\frac{1}{2})^n<\frac{\varepsilon}{8}$ and set $\mathcal{C}^k_k=\mathcal{D}_n$ and $\mathcal{C}^k_i=\mathcal{C}_{i,n}$ for each $i<k$.

We construct the families $\mathcal{C}_{1,j},\ldots,\mathcal{C}_{k-1,j},\mathcal{D}_j$ by induction on $j$. Suppose that they are constructed for a certain $j<n$. Let $m_j$ and $l_j$ ($l_j$ depending on $m_j$) be large enough.

Using Claim \ref{coveringlines} choose a family of $\chi$-augmenting cycles $\mathcal{A}$  such that a.e. interval of length at least $m_j$ is covered by both equivalence classes of $\chi$-augmenting cycles in $\mathcal{A}$ which are disjoint outside of $L(\chi)$ from those cycles of $\mathcal{C}_{1,j},\ldots,\mathcal{C}_{k-1,j},\mathcal{D}_j$ which intersect that interval. Next, find a family $\mathcal{I}$ of finite subintervals of $L(\chi)$ each of length at least $l_j$ and covering all but a small proportion of $L(\chi)$ such that the cycles from $\mathcal{A}\cup\mathcal{C}_{1,j}\cup\ldots\cup\mathcal{C}_{k-1,j}\cup\mathcal{D}_j$ that intersect distinct intervals from $\mathcal{I}$ are disjoint. 

\bigskip

\noindent Now, until further notice, we work locally on each $I\in\mathcal{I}$. 

We can assume that equivalent cycles in $\mathcal{A}$ do not substantially intersect outside of $L(\chi)$. This can be ensured, as we can merge such cycles and still cover the same edges on the line.

Next, apply Claim \ref{local.refinement} to each cycle $A\in\mathcal{A}$ intersecting $I$ to get a cycle $A'$ obtained by merging $A$ with at most two cycles $\{C_A^1,C_A^2\}$ from $\mathcal{C}_{1,j}\cup\ldots\cup\mathcal{C}_{k-1,j}\cup\mathcal{D}_j$. Note that by our assumption such cycles $A'$ still cover (in both equivalence classes) the same intervals of length $m_j$ as the original cycles $A$. Write $\mathcal{A}'$ for the family of such cycles $A'$. Note that if two equivalent cycles from $\mathcal{A}'$ cover disjoint subintervals of $I$, then they must be obtained by merging distinct cycles to members of $\mathcal{A}$, and thus must not have substantial intersections outside of the line $L(\chi)$, by our assumption on $\mathcal{A}$ and $\mathcal{C}_{1,j}\cup\ldots\cup\mathcal{C}_{k-1,j}\cup\mathcal{D}_j$.

Now, divide $\mathcal{A}'$ into two families  --- the intersection with the two equivalence classes of $\chi$-augmenting cycles. To each of those families apply Helly's theorem to shrink it so that each edge of the line $L(\chi)$ is covered by at most two $\chi$-augmenting cycles of this shrunken family. 
Next, shrink both families again by taking every other $\chi$-augmenting cycle in each of them and write $\mathcal{A}''$ for the family of $\chi$-augmenting cycles in the union of both shrunken families. Note that we can make sure that
 \begin{itemize}
 \item[(a)] the equivalent $\chi$-augmenting cycles from $\mathcal{A}''$ cover pairwise disjoint intervals of the line, 
 \item[(b)] each equivalence class of $\chi$-augmenting cycles in $\bigcup\mathcal{A}''$ covers at least a half of $I\setminus \bigcup(\mathcal{C}_{1,j}\cup\ldots\cup\mathcal{C}_{k-1,j}\cup\mathcal{D}_j)$.
 \end{itemize}
The condition (a) follows since we have two ways of taking every other $\chi$-augmenting cycle after applying Helly's theorem. The condition (b) can be arranged if $l_j$ is chosen large enough. Note that by (a) and the property of the family $\mathcal{A}$, we get that equivalent $\chi$-augmenting cycles from $\mathcal{A}''$ do not have substantial intersections outside of $L(\chi)$. 

\bigskip

\noindent Now we leave the interval $I\in\mathcal{I}$ and work globally again. 

For each $i<k$ shrink $\mathcal{C}_{i,j}$ to  $\mathcal{C}_{i,j+1}$ and $\mathcal{D}_j$ to $\mathcal{D}_j'$ by removing $\{C_A^1,C_A^2\}$ corresponding to the $\chi$-augmenting cycles in $\mathcal{A}''$ for all $I\in\mathcal{I}$. Note that the condition (i) can be arranged if $l_j$ is chosen large enough.
 

Finally, let $\mathcal{D}_{j+1}$ consist of the $\chi$-augmenting cycles from $\mathcal{A}''$ as well as the $\chi$-augmenting cycles from $\mathcal{D}_j$ whose interval does not intersect an equivalent $\chi$-augmenting cycle from $\mathcal{A}''$. Note that the condition (ii) can be arranged if $m_j$ is large enough.

\end{proof}

\section{Improving random perfect fractional matchings}
\label{section.final}

This sction is devoted to the proof of Theorem \ref{fmatching}, which will follow the same lines as the proof of Theorem \ref{uniformly.bounded} but will need one more technical lemma.

Note that $P_f$ is closed in $L^2(E(G))$ in the $\|\cdot\|_2$ norm.
We will consider $R_f$ with the induced $L^2$ distance. 
Note that $R_f$ is also closed in $L^2(E(G))$ with respect to the $\|\cdot\|_2$ norm. 
For the rest of this section, we consider $R_f$ as a Polish metric space with this induced $L^2$ metric.

Since the weak$^*$ topology on $P_f$ is also Polish and is weaker than the norm topology on $P_f$, both topologies induce the same Borel structures by \cite[Theorem 15.1]{kechris}. 
Thus, a Borel probability measure on the extreme points of $P_f$ is also a Borel probability measure on $R_f$ with the $L^2$ metric. We will consider the space of Borel probability measures on $R_f$ with the Wasserstein distance, induced by the $L^2$ metric on $R_f$.

\begin{lemma}\label{successor2}
Let $G$ be a hyperfinite bipartite one-ended graphing and let $f:V(G)\to\mathbb{N}$ be integrable.
Suppose there exists is a measurable  perfect fractional $f$-matching  $\tau:E(G)\to[0,1]$ such that $\supp(\tau)=G$.

Let $X$ be a random measurable perfect fractional $f$-matching such that a.s.
$X \in R_f$. 
Assume that $\mathbb{E}\mu(L(X))>0$. Let $\delta>0$. Then there exists a random measurable perfect fractional matching $X'$ which is a.s. an extreme point of $P_f$ and $\lambda\in[0,1]$ 
such that 
\begin{itemize}
    \item[(i)] $\mathbb{E} \mu(L(X'))< \mathbb{E}\mu(L(X))$
    \item[(ii)] $
W(X',X)\leq 3(\mathbb{E}\mu(L(X))-\mathbb{E}\mu(L(X')))$

\item[(iii)]  $\|\mathrm{bar}(X')-(\lambda\tau+(1-\lambda)\mathrm{bar}(X))\|_2\leq 
\delta(\mu(L(X)-\mathbb{E} \mu(L(X')))$
\end{itemize}

\end{lemma}

\begin{proof}
We will replace the random perfect fractional matching $X$ with a perfect fractional matching on a larger graphing. Consider the graphing $G^*$ with the vertex set $R_f\times V(G)$, where two vertices are adjacent if their first coordinates are equal and the second coordinates are adjacent in $G$. Note that $f$ and $\tau$ naturally extend to $f^*$ and $\tau^*$ on $V(G^*)$ and $E(G^*)$, by ignoring the first coordinates. Note that random perfect fractional $f$-matchings correspond on $G$ correspond to perfect fractional $f^*$-matchings on $G^*$. Write $\chi^*$ for the perfect fractional $f^*$-matching on $G^*$ corresponding to $X$. The measure $\mu$ on $E(G)$ and the measure $X$ on $R_f$ induce a measure on $E(G^*)$, which we denote by $\mu^*$. Note that $\mu^*(L(\chi^*))=\mathbb{E}(L(X))>0$ by our assumption.






 Now we apply Proposition \ref{successor1}, in case $L(\chi^*)$ is nowhere one-lined, and Proposition \ref{successor.oneline} in case $L(\chi^*)$ is one-lined in a set of components of positive measure, to obtain $\lambda\in[0,1]$ and a random perfect fractional $f^*$-matching $X^*$ on $G^*$ such that
\begin{itemize}
    \item[(i*)] $\mathbb{E}\mu^*(L(X^*))<\mu^*(L(\chi^*))$

    \item[(ii*)] $
      \mathbb{E}\mu^*(\{e\in E(G^*):X^*(e)\not=\chi^*(e)\})<3(\mu^*(L(\chi^*))-\mathbb{E}\mu^*(L(X^*)))$.

\item[(iii*)]  $\|\mathrm{bar}(X^*)-(\lambda\tau^*+(1-\lambda)\chi^*)\|_2\leq
\delta(\mu^*(L(\chi^*)-\mathbb{E} \mu^*(L(X^*)))$
\end{itemize}

Also, we set $X'$ to be the random perfect fractional $f$-matching on $G$ corresponding to $X^*$. This means that $X'$ is obtained by first choosing a random perfect fractional matching on $G^*$, distributed according to $X^*$, followed by randomly choosing $\chi$, distributed according to $X$ and then restricting the perfect fractional matching on $G^*$ to the $\chi$-th section of the graphing $G^*$. 


The condition (i) 
follows from (i*) and the fact that 
\begin{equation}\label{measures}
\mathbb{E}\mu(L(X'))=\mathbb{E}\mu^*(L(X^*)),\quad\mathbb{E}\mu(L(X))=\mu^*(L(\chi^*)).
\end{equation}

Note that the barycenter of $X'$ is obtained by taking the expectation of $\mathrm{bar}(X^*)$ over $\chi$. 
Note also that taking the expectation of $\chi^*$ over $\chi$ gives us the barycenter of $X$, 
and the expectation of $\tau^*$ over $\chi$ 
is equal to $\tau$. Thus (iii) follows from (iii*) and the Cauchy--Schwartz inequality: 

$$\int\limits_{E(G)}\big|\mathbb{E}_{\chi}\mathrm{bar}(X^*)-(\lambda\mathbb{E}_\chi\tau^*+(1-\lambda)\mathbb{E}_\chi\chi^*)\big|^2\leq\int\limits_{E(G)}\mathbb{E}_\chi [\mathrm{bar}(X^*)-(\lambda\tau^*+(1-\lambda)\chi^*)]^2.$$

For each $\chi\in R_f$ we write $X^*_\chi$ for the corresponding random perfect fractional $f$-matching on $G$. More precisely, for each $\chi\in R_f$ and a perfect fractional matching $\pi^*$ on $G^*$ we get the perfect fractional matching $\pi^*_\chi$ obtained by restricting $\pi^*$ to the $\chi$-th section of the graphing $G^*$. Then the measure $X^*_\chi$ is obtained as the pushforward of $X^*$ from $R_{f^*}$ to $R_f$ via the map $\pi^*\mapsto\pi^*_\chi$. Note that $X'$ can be represented by first randomly choosing $\chi$, distributed according to $X$, and then randomly choosing the perfect fractional matching, distributed according to $X^*_\chi$.

In order to obtain (ii) and bound $W(X,X')$, consider the measure $K$ on $R_f\times R_f$ which for $A\subseteq R_f\times R_f$ is defined as $K(A)=\int_{\chi\in R_f} X^*_\chi(A_\chi)\ dX(\chi)$ (here we write $A_\chi$ for the $\chi$-th vertical section of $A$).

Note that $K$ is a coupling of $X$ and $X'$. Now, 
\begin{align*}
   \int_{R_f\times R_f} \|\chi-\chi'\|_2 dK(\chi,\chi')= 
\int_{R_f} \mathbb{E}\|\chi-X^*_\chi\|_2 d X(\chi)\\
\leq\mathbb{E}\int_{R_f}\mu(\{e\in E(G):\chi(e)\not=X^*_\chi(e)\})dX(\chi)\\=\mathbb{E}\mu^*(\{e\in E(G^*):\chi^*(e)\not=X^*(e)\}),
\end{align*}
which, by (ii*) and (\ref{measures}) gives the estimate on $W(X,X')$.
\end{proof}

Now we are ready to prove our main Theorem \ref{fmatching}.

\begin{proof}[Proof of Theorem \ref{fmatching}]
By Remark \ref{constant1}, we can assume that $c$ is the constant function $1$ and $G=\supp{\tau}$ is nowhere two-ended..

By \cite[Lemma 3.21(i)]{jkl} on the union of infinite zero-ended components we have a measurable selector, so this union is a nullset. Thus, we can assume all zero-ended components are finite. For finite bipartite graphs the existence of a fractional perfect $f$-matching is equivalent to the existence of a perfect $f$-matching, so we can assume that $G$ is nowhere zero-ended, hence a.e. one-ended.


We will show that $\tau$ can be approximated by convex combinations of measurable perfect $f$-matchings in $\|\cdot\|_2$. This will be enough since a sequence convergent in $L^2(E(G))$ contains a subsequence convergent a.e. 

Denote by $M_f$ the set of all measurable perfect $f$-matchings on $G$. We consider $M_f$ with the $L^2$-metric. Note that $M_f\subseteq R_f$ is closed with respect to the $\|\cdot\|_2$ norm so it is a Polish space as well (and we are about to show that it is nonempty).

We will show that for every $\delta>0$ there exists a random perfect $f$-matching $X$ (i.e., a random perfect fractional $f$-matching $X$ concentrated on $M_f$) such that $\|\mathrm{bar}(X)-\tau\|_2<\delta$. To see that this is enough, note that $\mathrm{bar}(X)$ can be approximated in $\|\cdot\|_2$ by convex combinations of elements of $M_f$. Indeed, convex combinations of Dirac measures on $M_f$ are dense in the weak topology on $\mathcal{P}(M_f)$, so there is a sequence of such convex combinations converging to $X$ in $\mathcal{P}(M_f)$. This implies that the sequence converges to $X$ also in the weak topology of $\mathcal{P}(P_f)$. Then $\mathrm{bar}(X)$ is a limit of convex combinations of elements of $M_f$ in $P_f$. By Mazur's lemma there is a sequence of convex combinations of elements of $M_f$ converging to $\mathrm{bar}(X)$ in $\|\cdot\|_2$. 


Fix $\delta>0$. By the Choquet–-Bishop–-de Leeuw theorem, $\tau$ is the barycenter of a measure $X_0$ on the set $\mathrm{ext}(P_f)$ of extreme points of $P_f$. Starting with $X_0$, by induction, we define a sequence of random measurable perfect fractional $f$-matchings $X_\alpha$ concentrated on $ R_f$ for countable ordinals $\alpha$, so that

\begin{itemize}
    \item[(i)] $\mathbb{E}\mu(L(X_{\alpha+1}))<\mathbb{E}\mu(L(X_\alpha))$,
    \item[(ii)] $W(X_{\alpha+1},X_\alpha)<3(\mathbb{E}\mu(L(X_\alpha))-\mathbb{E}\mu(L(X_{\alpha+1})))$,
    \item[(iii)] $X_\gamma=\lim_{\alpha<\gamma} X_\gamma$ for any limit ordinal $\gamma<\omega_1$.
    \item[(iv)] $\|\mathrm{bar}(X_{n+1})-\tau\|_2\leq\|\mathrm{bar}(X_n)-\tau\|_2+\delta(\mathbb{E}\mu(L(X_n)-\mathbb{E} \mu(L(X_{n+1})))$
\end{itemize}

At every step $\alpha$ we will have 
\begin{itemize}
        \item[(v)] $\|\mathrm{bar}(X_\alpha)-\tau\|_2\leq\delta(\mathbb{E}\mu(L(X_0)-\mathbb{E} \mu(L(X_\alpha)))$
\end{itemize}



At successor stages, given $X_\alpha$, we use Lemma \ref{successor2} to find $X_{\alpha+1}$.



For limit ordinals $\gamma$ the condition (i) implies that the sequence $\mathbb{E}\mu(L(X_\alpha)$ for $\alpha<\gamma$ is strictly decreasing, and (ii) implies that $X_\alpha$ is Cauchy w.r.t. the Wasserstein distance. Thus, we can set $X_\gamma$ to be the weak limit of this sequence.
In particular, we have $X_\gamma=\lim_{\alpha<\gamma} X_\gamma$ in $\mathcal{P}(P_f)$, so $\mathrm{bar}(X_\gamma)=\lim_{\alpha<\gamma}\mathrm{bar}(X_\alpha)$ in $P_f$.
Note that we have $\mathbb{E}\mu(L(X_\gamma))=\lim_{\alpha<\gamma}\mathbb{E}\mu(L(X_\alpha))$ since $\chi\mapsto\mu(L(\chi))$ is continuous on $R_f$ with the $L^2$ distance. Finally, as $\|\cdot\|_2$ balls are closed in the weak$^*$ topology, (v) follows.



Note that there must exist $\alpha<\omega_1$, such that have $\mathbb{E}\mu(L(X_\alpha))=0$. This implies that the random perfect fractional $f$-matching $X_\alpha$ is a.s. a measurable perfect $f$-matching a.e. Finally, the condition (v) implies that $\|\mathrm{bar}(X_\alpha)-\tau\|_2\leq \delta\mathbb{E}\mu(L(X_0))\leq\delta$, which ends the proof.

\end{proof}

\begin{remark}\label{barycenter.regular}
Under the assumptions of Theorem \ref{uniformly.bounded} (so in particular for regular graphings), the proof of Theorem \ref{fmatching} above shows that $\tau$ can be made equal to a barycenter of a probability measure concentrated on measurable fractional perfect $f$-matchings. This is because under the assumptions of Theorem \ref{uniformly.bounded}, in Lemma \ref{successor2} above in place of (iii) we get that $\mathrm{bar}(X')$ is equal to a convex combination of $\mathrm{bar}(X)$ and $\tau$, since we can use Proposition \ref{successor1} only and do not need to refer to Proposition \ref{successor.oneline}. 
\end{remark}

\section{Factor of iid perfect matchings}
\label{sec:factor-iid-perfect}

In this section we prove Theorem~\ref{fiid}. The following will be the key lemma in the two-ended case.

\begin{lemma}\label{HxZ}
Consider a group $\Gamma=\mathbb{Z}\ltimes\Delta$, where $\Delta$ is a finite group, with a finite symmetric generating set for $\Gamma$ such that the corresponding Cayley graph is bipartite.
\begin{itemize}
    \item[(i)] If $\Delta$ has odd order and $\Gamma\curvearrowright (V,\nu)$ is an a.e. free and totally ergodic action on a standard probability space, then its Schreier graphing does not admit a measurable perfect matching,
    \item[(ii)] If $\Delta$ has even order and $\Gamma\curvearrowright V$ is a free Borel action, then its Schreier graphing admits a Borel perfect matching.
\end{itemize}

\end{lemma}


\begin{proof}
Since $\Gamma\slash\Delta=\mathbb{Z}$, there exists an element $\gamma\in\Gamma$ such that $\gamma\Delta$ generates this quotient. Note that $\mathbb{Z}$ acts on $\Gamma$ and the Cayley graph of $\Gamma$ is a disjoint union of copies of $\Delta$ arranged in a bi-infinite line $\bigcup_{i=-\infty}^\infty\gamma^i\Delta$. Also, since the cyclic group generated by $\gamma$ is of finite index, it has a further finite index subgroup $Z$, generated by $\gamma^n$ for some $n>0$, such that $Z$ is normal in $\Gamma$. 

\textbf{(i)}. First, consider the case when $|\Delta|$ is odd. 
Recall the following folklore fact. 
\begin{claim}\label{folklore}
  If $G$ is the Schreier graphing of a  free totally ergodic action of $\mathbb{Z}$ with respect to $\pm1$ as generating set, then $G$ does not admit a measurable perfect matching. 
\end{claim}
\begin{proof} 
Suppose $G$ admits a measurable perfect matching. The set of vertices $A=\{x\in V(G): x \text{ is matched to } x+1\}$
is invariant under $2\mathbb{Z}$. Thus, it is either null or co-null by total ergodicity. However $V(G)\setminus A$ is obtained by shifting $A$ by $1$.
\end{proof}
Now, let $\Gamma\curvearrowright (V,\nu)$ be a.e. free and totally ergodic. Assume $M$ is a measurable perfect matching. Note that on each orbit (which we identify with the Cayley graph of $\Gamma$),  for every $i$ exactly one of the two sets of edges 
$E(\bigcup_{j=-\infty}^i \gamma^{i} \Delta, \bigcup_{j=i+1}^{\infty} \gamma^j \Delta) \cap M$ or $E(\bigcup_{j=-\infty}^{i-1} \gamma^{i} \Delta, \bigcup_{j=i}^{\infty} \gamma^j \Delta) \cap M$ has odd size. Write $(V',\nu')$ for the quotient of $V$ by the action of the group $\Delta$ and note that $(V',\nu')$ is also a standard probability space. Note that the action $\Gamma\curvearrowright (V,\nu)$ induces the action $\mathbb{Z}\curvearrowright (V',\nu')$ and we consider the Schreier graphing $G'$ of the action $\mathbb{Z}\curvearrowright (V',\nu')$ with respect to the generators $\pm1$. Consider the set $M'$ of edges of $G'$ that on every orbit contains $(i,i+1)$ if $E(\bigcup_{j=-\infty}^i \gamma^{i} \Delta, \bigcup_{j=i+1}^{\infty} \gamma^j \Delta) \cap M$ is odd. Note that $M'$ is a measurable perfect matching in $G'$. The action $\mathbb{Z}\curvearrowright (V',\nu')$ is totally ergodic since $\Gamma\curvearrowright (V,\nu)$ is totally ergodic. This gives a contradiction by Claim \ref{folklore}.

 
 \medskip

\textbf{(ii)}.
Now assume that $|\Delta|$ is even. We consider a bipartite Schreier graph $G$ of a Borel action $\Gamma\curvearrowright V$, induced by a set of generators $\Sigma\subseteq\Gamma$. 


\begin{claim}\label{half}
 In the Cayley graph of $\Gamma$ either $\Delta$ is entirely contained in one of the class of the bipartition, or it contains an equal number of vertices from both classes.
\end{claim}
\begin{proof}
 Note that the set of elements of $\Delta$ which have even word length in the Cayley graph forms a subgroup $\Delta'$ of $\Delta$. Elements of $\Delta'$ lie in the same class of the bipartition as the neutral element, while the elements of $\Delta\setminus\Delta'$ lie in the other class. As $\Delta'$ is a subgroup of index at most $2$, either it is equal to $\Delta$ or constitutes half of its elements.
\end{proof}

Write $V'$ for the quotient of $V$ by the action of the group $\Delta$ and note that $V'$ is also a standard Borel space. Recall that $\gamma\in\Sigma$ is a generator of $\Gamma$ such that $\Gamma=\bigcup_{i=-\infty}^\infty\gamma^i\Delta$. 
Consider 
the graphing $G'$ on $V'$ induced by the generator $\gamma$. 
Let $m=\max \{i: \exists \sigma \in \Sigma, 
\sigma\Delta=\gamma^i\Delta \}$. 
Note then that for every generator  $
\sigma \in \Sigma\setminus\Delta$ we have $\sigma\Delta=\gamma^l\Delta$ for some $0<l\leq m$. 
\begin{claim}\label{blocks}
There exists a finite subset $\Phi\subseteq \Gamma$ containing the identity 
such that 
\begin{itemize}
    \item the induced Cayley graph on $\Phi$ admits a perfect matching, 
    \item the induced Cayley graph on the complement of $\Phi$ in $\Gamma$ admits a perfect matching,
    \item the complement of $\Phi$ in $\Gamma$ consists of two infinite connected components.
    \item if $k$ is such that $\Phi\cap \gamma^k\Phi=
    \emptyset$, then the complement of $\Phi\cup \gamma^k\Phi$ consists of two infinite and one finite component, and the latter has half of its vertices in each class of the bipartition.
\end{itemize}


\end{claim}

\begin{proof}
By Claim \ref{half}, we consider two cases.

\textbf{Case 1.} Suppose that each copy of $\Delta$ in $G$ contains an equal number of vertices from each class of the bipartition. Then we choose $\Phi$ to contain 
$2lm$ copies of $\Delta$ starting from the identity.  We divide $\Phi$ into $m$ consecutive unions of $2l$ consecutive copies  of $\Delta$. The perfect matching on each such consecutive union of copies of $\Delta$ is induced by $\sigma$. Note that it can be extended to the left and to the right also using $\sigma$ in the same way.  

\begin{figure}[ht]
    \centering
    \includegraphics[width=0.55\textwidth]{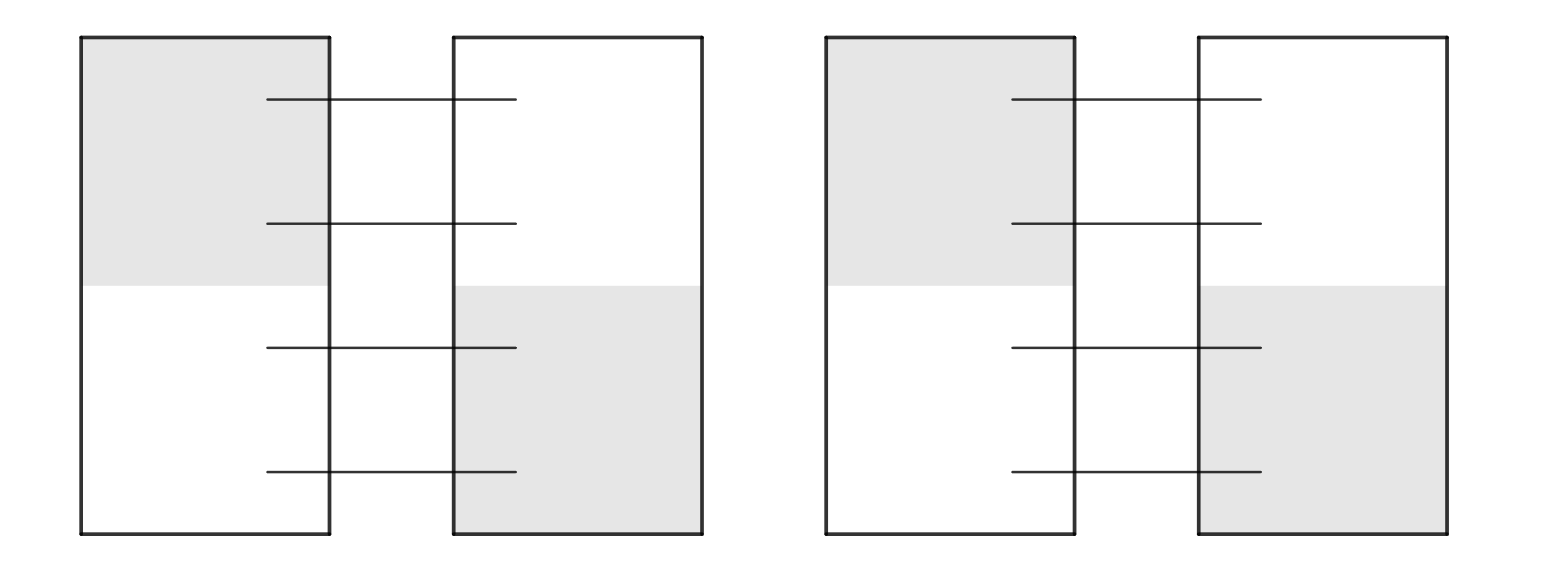}
    \caption{Perfect matching on $\Phi$ in Case 1 when $l=1$, $m=2$}
    \label{fig:6.4.1}
\end{figure}

If $k$ is such that $\Phi\cap \gamma^k\Phi=\emptyset$, then the finite component of the complement of $\Phi\cap \gamma^k\Phi=\emptyset$ is a union of copies of $\Delta$, so contains an equal number of elements of each class of the bipartition.


\textbf{Case 2.} Suppose that each copy of $\Delta$ is contained in one class of the bipartition. 
Then we choose $\Phi$ to contain 
$2lm$ consecutive copies of $\Delta$, starting at the identity, together with $l$ halves of copies of $\Delta$ preceding them as well as $l$ halves of copies of $\Delta$ following them. 
The perfect matching on $\Phi$ is induced by $\sigma$ so that on every copy of $\Delta$ intersecting the block
half of the vertices is matched in the $\sigma$ direction and half in the $\sigma^{-1}$ direction (in the first $l$ halves we only use the $\sigma$ direction and in the last $l$ halves we only use the $\sigma^{-1}$ direction). Note that this matching can be extended to the left and to the right to a perfect matching in the same way.  

\begin{figure}[ht]
    \centering
    \includegraphics[width=0.75\textwidth]{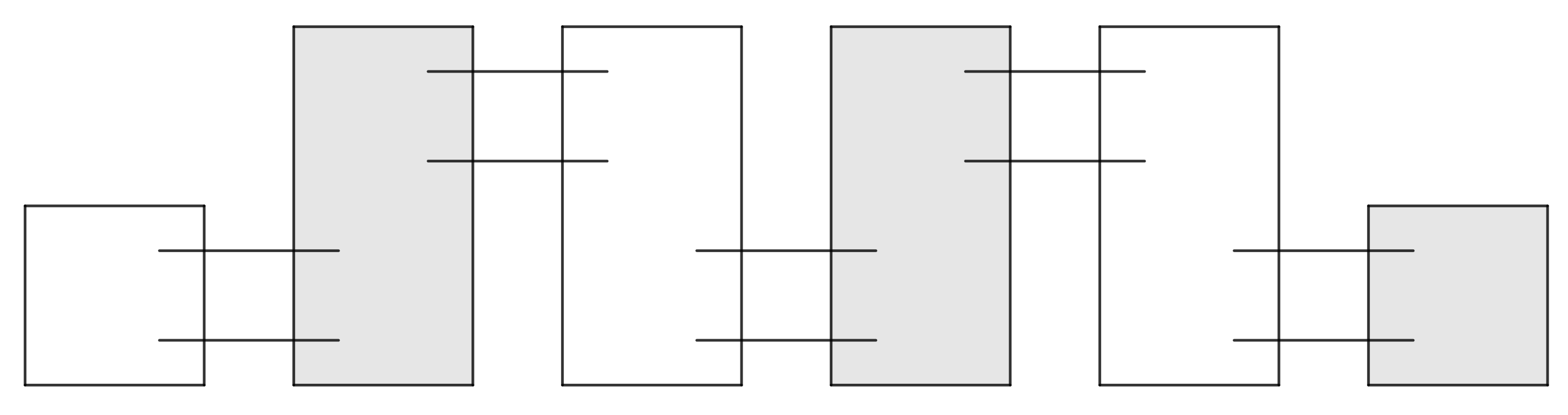}
    \caption{Perfect matching on $\Phi$ in Case 2 when $l=1$, $m=2$}
    \label{fig:6.4.1}
\end{figure}

Suppose that $k$ is such that $\Phi\cap \gamma^k\Phi=\emptyset$.
To see that the finite component of the complement of $\Phi\cup \gamma^k\Phi=\emptyset$
contains an equal number of elements of both parts of the bipartition, note that it is the union of fragments of $\sigma$-orbits on the copies of $\Delta$, and on each such fragment the classes of the bipartition alternate between the copies of $\Delta$. In each such fragment, the vertices in the halves of the copies of $\Delta$ on the boundary of the component of $\Phi\cup\gamma^k\Phi$ make up the difference in the number of vertices in each class of the bipartition.

\end{proof}

Let $\Phi$ be such as in Claim \ref{blocks} and $k_0$ be such that whenever $k>k_0$, then $\Phi\cap\gamma^k\Phi=\emptyset$. For $k>k_0$ sufficiently large, to be determined later,  construct a maximal Borel 
$k$-discrete set $T'$ in the graphing $G'$ (i.e., any of its two vertices are of distance at least $k$ and $T'$ is maximal with this property, see \cite[Lemma 7.3]{kechris.miller}). 
Consider the corresponding collection $T$ of representatives of copies of $\Delta$ on orbits in $V$ and let $B=\Phi\cdot T$.

Note that on each component of $G$, the induced Cayley graph on $B$ has finite components, such that each such finite component has a perfect matching, and its complement also has a perfect matching. Also, the induced Cayley graph on the complement of $B$ has finite components which contain an equal number of elements from each class of the bipartition. 
We will refer to the finite components of $B$ as to \textit{blocks} and to the finite components of the complemement of $B$ as to \textit{gaps}. We clearly have a Borel perfect matching on the union of blocks.
If we can extend this perfect matching in a Borel way to the gaps, then we will be done.



Given a gap, we use the convention to enumerate full copies of $\Delta$ contained in the gap (say from left to right) as $\gamma^1\Delta,\gamma^2\Delta,\ldots,\gamma^t\Delta$ (for some $t$). First, we need the following claim. 

Recall that $n$ was chosen so that $\gamma^n$ generates a normal subgroup of $\Gamma$.

\begin{claim}
 Fix a gap $C$ and let $H\subseteq C$ be contained in one class of the bipartition. If $n_1<n_2$ with 
$n_2-n_1>m$ and $n\,|\,n_2-n_1$ are
such that $$\gamma^{n_2-n_1}(H \cap \bigcup_{j=n_1}^{n_1+2m} \gamma^j \Delta) = H \cap \bigcup_{j=n_2}^{n_2+2m} \gamma^j \Delta$$ and $H\cap \bigcup_{j=n_1+m}^{n_2+m-1} \gamma^j \Delta$ is a proper nonempty subset of $\bigcup_{j=n_1+m}^{n_2+m-1} \gamma^j \Delta$, then $$|N_G(H)\cap \bigcup_{j=n_1+m}^{n_2+m-1} \gamma^j \Delta|>|H\cap \bigcup_{j=n_1+m}^{n_2+m-1} \gamma^j \Delta |.$$
\end{claim}
\begin{proof}
Write $J=\bigcup_{j=n_1+m}^{n_2+m-1} \gamma^j \Delta$. Since $n\,|\,n_2-n_1$, we get that $\langle \gamma^{n_2-n_1} \rangle\lhd\Gamma$.
Consider the Cayley graph $F$ of $\Gamma/ \langle \gamma^{n_2-n_1} \rangle$ and note that $F$ is a regular finite bipartite graph, so it admits a perfect matching. This implies that $|N_G(H \cap J)| \geq |H \cap J|$ holds.

However, if $|N_G(H \cap J)| = |H\cap J|$ then the Hall condition
would be sharp for a proper nonempty subset in $F$, and this contradicts the connectivity of  $F$.

So we get that that  $|N_G(H \cap J)| > |H\cap J|$.
\end{proof}
Now we show how to extend the perfect matching to the gaps.

\begin{claim}\label{gaps}
  If $k$ is large enough, then each gap admits a perfect matching.
\end{claim}

\begin{proof}

 Let $C$ be a gap. In order to show that $C$ admits a perfect matching, we have to check Hall's condition for $C$. Note that since $G$ satisfies Hall's condition, the only way that $C$ may fail is from the loss of edges leaving $C$. There is an upper bound $b=|\partial\Phi|$  on the number of such edges that is independent of the size of $C$.  We will show that as long as $k$ is large enough this loss is overcome by extra edges between the inner copies of $\Delta$ in $C$. Consider a subset $H\subseteq C$ contained in one class of the bipartition, and suppose for a contradiction that $|N_C(H)|<|H|$.

If $k$ is large enough, then we choose $n_1<n_2<\ldots<n_b$ such that
$m<n_{i+1}-n_i$ and $n\,|\,n_{i+1}-n_i$ and
 $$\gamma^{n_{i+1}-n_i}(H \cap \bigcup_{j=n_i}^{n_i+2m} \gamma^j \Delta) = H \cap \bigcup_{j=n_{i+1}}^{n_{i+1}+2m} \gamma^j \Delta.$$
 Write $J_i=\bigcup_{j=n_i+m}^{n_{i+1}+m-1} \gamma^j \Delta$.
 
 Note that $H$ intersects each $J_i$, because otherwise $H$ would be
partitioned into two sets far from each other such that one of them does not satisfy the Hall condition, contradicting the fact
that the set of all vertices to the left or right of a given block can be extended to a perfect matching. Similarly, $H$ can not contain its entire class in $J$,
as we could apply the same argument to the complement of $N_C(H)$ in the other class, using the fact that $C$ has the same number of elements from both classes of the bipartition.

This implies that $|N_G(H \cap J_i)|>|H\cap J_i|$ for each $i=1,\ldots,b$, making up for the loss of edges leaving $C$, which means that $|N_G(H)|>|H|$, as needed.

\end{proof}
Choosing $k$ big enough as in Claim \ref{gaps}, we can construct a Borel perfect matching in $G$, which ends the proof.


\end{proof}

In the amenable case of Theorem \ref{fiid}, we will use the following proposition.

\begin{prop}\label{new}
Let $\Gamma$ be a finitely generated amenable group and let $G$ be a bipartite Schreier graphing of an a.e. free 
action of $\Gamma$ 
\begin{itemize}
    \item If $\Gamma$ is isomorphic to $\mathbb{Z} \ltimes \Delta$ for a finite normal subgroup $\Delta$ of odd order and the action is totally ergodic, then $G$ does not admit a measurable perfect matching.

    \item If $\Gamma$ is not isomorphic to $\mathbb{Z} \ltimes \Delta$ 
    for a finite normal subgroup $\Delta$ of odd order, then $G$ admits a measurable perfect matching.
\end{itemize}
\end{prop}

\begin{proof}
If $\Gamma$ is amenable and one-ended then the existence of a measurable perfect matching follows from Corollary~\ref{regular}. 

Thus, we need to consider the two-ended case. Recall (see \cite[Theorem 5.12]{scott.wall}) that there exists a finite normal subgroup $\Delta$ such that $\Gamma/\Delta$ is either isomorphic to $\mathbb{Z}$ or to $\zz$. In the first case, $\Gamma$ is isomorphic to $\mathbb{Z}\ltimes\Delta$ for a finite group $\Delta$, and Lemma~\ref{HxZ} implies the theorem.

Now assume that $\Gamma/\Delta$ is isomorphic to $\zz$. Note that  the group $\zz$ has a normal subgroup of index two isomorphic to $\mathbb{Z}$, and the elements not in this subgroup have order two. Hence there exists a generator $\gamma$ defining the Cayley graph of $\Gamma$ whose image in $\Gamma/\Delta$ has order two, hence $\gamma$ has finite, even order. The Cayley graph of $\Gamma$ decomposes into $\gamma$-cycles of even length and the same is true for the graphing $G$, so it admits a measurable perfect matching.


\end{proof}

Now we prove Theorem \ref{fiid}.
\begin{proof}[Proof of Theorem \ref{fiid}]
If $\Gamma$ is non-amenable then the Lyons-Nazarov theorem \cite{lyons.nazarov} shows that there exists a factor of iid perfect matching. If $\Gamma$ is amenable, then the statement follows from Proposition \ref{new} since the Bernoulli shift is a.e. free and totally ergodic.

\end{proof}

\section{Measurable circle squaring}\label{sec:mcs}

Recall that the circle squaring
problem of Tarski \cite{Tarski.problem} asked if the unit disc and the unit square are equidecomposable. A positive answer was given by
Laczkovich \cite{lacz11} and recently Grabowski, M\'ath\'e and Pikhurko \cite{gmp} have given a measurable equidecomposition. We show a couple of different ways to deduce the measurable circle squaring from the main results of this paper. 

In applications to equidecomposition problems, we consider measurably bipartite graphs. We will use the following notation. Given two disjoint sets $A,B$ of a metric space write $G^k(A,B)$ for the bipartite graph whose vertices are $A\cup B$ and edges join two points, one from $A$ and one from $B$ at distance at most $k$. In particular, if $\Gamma$ is a finitely generated group and $d$ is the metric on its Cayley graph with respect to a finite generating set $\Sigma\subseteq\Gamma$, then for two disjoint subsets $A,B\subseteq \Gamma$ we denote the above bipartite graph by $G^k_\Sigma(A,B)$, 
Similarly, if $\Gamma\curvearrowright X$ is a free action, $\Sigma\subseteq\Gamma$ is a finite generating set and $A,B\subseteq X$ are disjoint, then we consider the bipartite graph (measurably bipartite if $A,B$ are measurable) $G^k_{\Gamma\curvearrowright X,
\Sigma}(A,B)$ whose vertices are $A\cup B$ and edges join two points in the same orbit if their distance in the Schreier graph is at most $k$.

We will be interested in the graph $G^k_{\mathbb{Z}^d\curvearrowright\mathbb{T}^n,\Sigma}(A,B)$ obtained from the action of $\mathbb{Z}^d$ on the torus $\mathbb{T}^n$ constructed by Laczkovich \cite{lacz11}, for a suitably chosen $k$, where $A,B\subseteq\mathbb{T}^n$ and $\Sigma$ is the standard set of generators of $\mathbb{Z}^d$.

\subsection{A proof of the measurable circle squaring using regular bipartite graphs}

In this proof we directly use our perfect matching theorem \ref{matching}. We show that the graph corresponding to the action of $\mathbb{Z}^d$ constructed by Laczkovich \cite{lacz11} is very close to  a 
regular graph. More precisely, using his construction, we consider a regular subgraph of the Schreier graphing of $\mathbb{Z}^d$. Even though that graph is not measurable, we use a sandwich-type argument to approximate it with measurable graphings and apply Theorem \ref{fmatching}.


Recall \cite{lacz.uniform} that a set $A\subseteq\mathbb{R}^d$ is \textit{uniformly spread with density $\alpha$} 
if 
there exists a bijection $h:A\to(\frac{1}{\sqrt[d]{\alpha}}\mathbb{Z})^d$ such that $\sup_{x\in A}|h(x)-x|<\infty$. Uniformly spread sets can be defined in many equivalent ways \cite[Theorem 1.1]{lacz.uniform}. In the proof of the circle squaring, Laczkovich constructs an action of $\mathbb{Z}^d$ so that both the unit disc and the unit square are uniformly spread on each orbit. More precisely, given a free pmp action $\mathbb{Z}^d\curvearrowright (X,\nu)$ for each $x\in X$ we identify the orbit $O_x$ of $x$ with $\mathbb{Z}^d$ in a canonical way. We say that a measurable set $A\subseteq X$ is \textit{uniformly spread with density $\alpha$} (with respect to the action $\mathbb{Z}^d\curvearrowright (X,\nu)$) if for each $x\in X$ there is a bijection $h_x$ between $A\cap O_x$ (identified with a subset of $\mathbb{R}^d$) and $\frac{1}{\alpha}\mathbb{Z}^d$ such that $\sup_{x\in X}\sup_{y\in O_x} |h_x(y)-y|<\infty$. The Laczkovich theorem (more precisely, the combination of the Laczkovich Bijection Lemma \cite[Lemma 9.8]{wagon} and \cite[Theorem 9.12, Claim 2]{wagon}) shows that the measurable sets in the circle squaring are uniformly spread with respect to an action of $\mathbb{Z}^d$.


Below, we write $\Sigma\subseteq\mathbb{Z}^d$ for the standard set of generators (the unit vectors). Given $A,B\subseteq\mathbb{Z}^d$ we write $G^k(A,B)$ for $G^k_\Sigma(A,B)$.
In the following proposition we construct a perfect fractional matching that is not necessarily measurable but it is defined on a measurable one-ended graphing and its support contains a measurable one-ended graphing.

\begin{prop}\label{propozycja}
Let $d>1,\alpha\in(0,1)$. Consider a free pmp action $\mathbb{Z}^d\curvearrowright (X,\nu)$ and suppose that $A,B\subseteq X$ are measurable and uniformly spread with density $\alpha$. 
Then there exist $k<l$ 
such that 
the graph $G^k(A,B)$ is one-ended and $G^l(A,B)$ admits a perfect fractional matching $\phi$ that is constant and positive 
on the edges of $G^k(A,B)$.

\end{prop}

\begin{proof}
First, fix an orbit $O$ of the action of $\mathbb{Z}^d$.
Note that the embedding of $\mathbb{Z}^d$ into $\mathbb{R}^d$ is a quasi-isometry and on $\mathbb{Z}^d$ the natural Cayley graph metric is equivalent to the Euclidean metric (induced by this embedding). Below, slightly abusing the notation, for subsets $A',B'\subseteq \mathbb{R}^d$ we write $G^m(A',B')$ for the graph with edges $(a',b')$ ($a'\in A',b'\in B'$) if the Euclidean distance of $a'$ and $b'$ is at most $m$.  

There is a perfect matching between $A\cap O$ and $A'=\frac{1}{\alpha}\mathbb{Z}^d$, and between $B\cap O$ and  $B'=\frac{1}{\alpha}\mathbb{Z}^d+(\frac{\alpha}{2},\ldots,\frac{\alpha}{2})$. Both matchings match each vertex $x$ to a vertex $y$ such that $|x-y|$ is bounded by a constant, since we assume that $A\cap O$ and $B\cap O$ are uniformly spread.

Note that the graph $G^m(A',B')$ is regular for every $m$. The matching between $A\cap O$ and $A'$ allows us to identify the elements of $A$ and $A'$, and similarly, the matching between $B\cap O$ and $B'$ allows us to identify the elements of $B$ and $B'$. We obtain using this identification that for every $m$ there exists $m^+>m$ such that
all edges in $G^m(A,B)$ are still edges in $G^{m^+}(A',B')$, and similarly, all edges in $G^m(A',B')$ are still edges in $G^{m^+}(A,B)$.

Choose $n$ large enough 
such that $G^{n}(A',B')$ is one-ended. Put $n_1=n^+$, $n_2=n_1{}^+$, 
$n_3=n_2{}^+$. 

Note that the graph $G^{n_2}(A',B')$ is regular and let $r$ be its degree.
Now, we can transport the constant $\frac{1}{r}$ fractional matching on $G^{n_2}(A',B')$ to $G^{n_3}(A,B)$ using the identification of $A$ with $A'$ and $B$ with $B'$ (i.e. we put $\frac{1}{r}$ on every edge between $a\in A\cap O$ and $b\in B\cap O$ if $a$ is identified with $a'\in A'$, $b$ is identified with $b'\in B'$ and there is an edge between $a'$ and $b'$). Note that it is equal to $\frac{1}{r}$ on all edges from $G^{n_1}(A,B)$. $G^{n_1}(A,B)$ is also one-ended, as it contains a subgraph quasi-isometric to $G^n(A',B')$ via the identification of $A$ with $A'$ and $B$ with $B'$. Finally, put $k=n_1$, $l=n_3$. 

Note that since we do not require that the perfect fractional matching is measurable, we can define it separately on each orbit $O$. However, the choice of $k$ and $l$ does not depend on the choice of the orbit $O$ since we assumed that $A$ and $B$ are uniformly spread with respect to the action $\mathbb{Z}^d\curvearrowright (X,\nu)$.


\end{proof}

\begin{theorem}\label{mcs1}
Let $d>1,\alpha\in(0,1)$. Consider a free pmp action $\mathbb{Z}^d\curvearrowright (X,\nu)$ and $A,B\subseteq X$ measurable and uniformly spread with density $\alpha$. 
Then $A$ and $B$ are equidecomposable using measurable pieces.
\end{theorem}
\begin{proof}
 By Proposition \ref{propozycja}, there exist $k<l$ such that writing $H=G^k_{\mathbb{Z}^d\curvearrowright X,\Sigma}(A,B)$ and $G=G^l_{\mathbb{Z}^d\curvearrowright X,\Sigma}(A,B)$, the graphing $H$ is one-ended and on $G$ there exists a perfect fractional matching between $A$ and $B$ such that on the edges of $H$, it is constant and positive. By Lemma \ref{lemat} and Theorem \ref{fmatching} the graphing $G$ admits a measurable perfect matching, so we are done.
\end{proof}

\begin{cor}[Measurable circle squaring]
  Let $A,B\subseteq \mathbb{T}^n$ be such that $\nu(A)=\nu(B)>0$ and $\dim(\partial A),\dim(\partial B)<n$. Then $A$ and $B$ are measurably equidecomposable.
\end{cor}
\begin{proof}
By \cite[Theorem 9.12, Claim 2]{wagon} and \cite[Lemma 9.8]{wagon}, there exists $d>1$, $\alpha>0$ and a free pmp action $\mathbb{Z}^d\curvearrowright\mathbb{T}^n$ such that 
the sets $A$ and $B$ are uniformly spread with density $\alpha$. We are done by Theorem \ref{mcs1}.
\end{proof}

\subsection{A proof of the measurable circle squaring using two independent equidecompositions}

In this proof of the measurable circle squaring we use Theorem \ref{fmatching} and the fact that circle squaring is possible by a large set of random translations (with high probability), as proved by Laczkovich.

\begin{prop}\label{circlesquaring}
Let $\Gamma$ be an amenable group, $\Gamma_1, \Gamma_2 \leq \Gamma$ be its subgroups and $\Gamma\curvearrowright X$ be a probability measure preserving free action of $\Gamma$ on a probability measure space $X$. Assume that for every $\gamma \in \Gamma$ the intersection $\Gamma_1 \cap \gamma^{-1} \Gamma_2 \gamma$ is finite. Let $A, B \subseteq X$ be measurable and $\mu(A),\mu(B)>0$. If $A$ and $B$ admit an equidecomposition both by $\Gamma_1$ and $\Gamma_2$ then they admit a measurable $\Gamma$-equidecomposition.
\end{prop}


\begin{proof}
We may assume that the groups are finitely generated and $\Gamma$ is the subgroup generated by $\Gamma_1 \cup \Gamma_2$. Let $\Sigma_i$ denote the set of generators of $\Gamma_i$ for $i=1,2$.
Let $k$ be such that $A$ and $B$ are equidecomposable using elements of $\Gamma_i$ of length at most $k$ with respect to $\Sigma_i$. For $i=1,2$ consider the bipartite graphings $G_i=G^k_{\Gamma_i\curvearrowright X,\Sigma_i}(A,B)$, each with vertex set $A \cup B$. 
We apply Lemma \ref{lemat} to the graphings $G_i$ (with the subgraphs $H_i$ being empty) to get measurable perfect fractional matchings $\tau_1$ and $\tau_2$ of $A$ and $B$ in both graphings $G_1$ and $G_2$ respectively. 

Now, let $H_i$ be the spanning subgraph of $G_i$ that consists of the edges where $\tau_i$ is positive. 
And let $H$ denote the spanning subgraph with edge set $E(H)=E(H_1) \cup E(H_2)$. We will prove that $A$ and $B$ admit a measurable perfect matching in $H$. Note that $\tau=\frac{1}{2}\tau_1+\frac{1}{2}\tau_2$ is a measurable perfect fractional matching which is positive on all edges in $H$. 

First, for $i=1,2$ we show that the $H$-components that contain infinitely many infinite $H_i$-components are one-ended. By Proposition~\ref{equi} it suffices to show that the growth of the balls in these components is superlinear. Given a vertex $x$ and $C>0$ if the radius $r$ is large enough then $B(x,r)$ intersects more than $C$ infinite $H_i$-components, 
hence $|B(x,r+1) \setminus B(x,r)|>C$. Since this holds for every $C$ these components have superlinear growth. We can apply Theorem~\ref{fmatching} to the union of these $H$-components and $\tau$ in order to get a measurable perfect matching.

Next, for $i=1,2$ consider the union of the $H$-components that do not contain an infinite $H_i$-component. These have only finite $H_i$-components with a measurable perfect fractional matching, hence they admit a measurable perfect matching. 

Finally, we consider the union of the $H$-components such that for both $i=1$ and $i=2$ they contain infinite $H_i$-components, but only finitely many of them.
We show that this union is a nullset.
Consider the set of vertices $S$ in the union of these components that are in infinite $H_1$-components and their distance from infinite 
$H_2$-components is minimal in the component. The set $S$ is measurable. For every $\gamma \in \Gamma, H_1$-component $C_1$ and $H_2$-component $C_2$ 
the set of vertices $x \in H_1$ such that $\gamma x \in H_2$ is a finite set. Hence $S$ intersects every such infinite $H$-component in a finite set. Since the set $S$ is contained by the union of such infinite $H$-components, $S$ is a nullset. Hence the union of these $H$-components is also a nullset. This ends the proof.
\end{proof}

This gives an alternative proof for the measurable squaring of the disc using Laczkovich's result \cite{lacz2}.
In particular, this shows that if $d$ random vectors suffice for the circle squaring then $2d$ random vectors suffice for the measurable circle squaring.

\begin{cor}[Measurable circle squaring]\label{mcsnumber}
  Let $A,B\subseteq \mathbb{T}^n$ be measurable sets that are equidecomposable by a random set of $d$ vectors chosen uniformly and independently in $\mathbb{T}^n$ with positive probability. 
  Then $A$ and $B$ are measurably equidecomposable by a random set of $2d$ vectors chosen uniformly and independently in $\mathbb{T}^n$ with positive probability. 
\end{cor}
\begin{proof}
We apply Proposition \ref{circlesquaring} with the following choice. Laczkovich proved \cite{lacz2} that if $d$ is large enough then for the random set $u_1,\ldots,u_d$ of $d$ vectors chosen uniformly and independently in $\mathbb{T}^n$ with positive probability, the unit square and the unit disc are equidecomposable via the action of $\mathbb{Z}^d$ on the torus. We choose $2d$ random vectors $u_1,\ldots,u_{2d}$ and consider the corresponding action of $\Gamma=\mathbb{Z}^{2d}$, where we consider $\Gamma_1$ generated by the first $d$ generators of $\Gamma$ and $\Gamma_2$ generated by the last $d$ generators of $\Gamma$.
\end{proof}


\section{Appendix}
\label{sec:appendix}

\begin{proof}[Proof of Proposition \ref{equi}]
\mbox{}

\textbf{(i) $\rightarrow$ (ii)}. Assume that $G$ admits linear growth a.e.
Recall that every graphing has a.e. zero, one, two or infinite ends \cite[Theorem 5.2]{adams}. Components with zero ends are finite, and so the union of these components is a nullset, 
since $G$ has a.e. linear growth. 
We need to show that the components which are not two-ended also form a nullset. By \cite[Theorem 2.1]{cgmtd}, 
 there exists a one-ended spanning subforest in the union of not two-ended components.
 
 Note that this subforest also has linear growth and spans the same set of vertices, so we may in fact assume that every one-ended component is a one-ended tree directed towards the end. 


Let $S_0$ be the set of leaves and $S_k$ denote the set of vertices such that the finite subtree cut by those vertices has height $k$. Given $C>0$ let $R_C$ denote the following set of vertices: a vertex $x$ is in $R_C$ if there exists $k$ and  $y \in S_k$ such that there is a path directed towards the end from $x$ to $y$ and the number of vertices that are below $y$ is at least $Ck$. We will show that $\nu(R_C)=1$ for every $C$. This is sufficient, since then the vertices in $\bigcap_{C=1}^{\infty} R_c$ have superlinear growth, and this set is non-empty, moreover, conull. Suppose for a contradiction that there exists a $C$ such that $\nu(R_C)<1$. The induced subgraphing $H$ on $V(G) \setminus R_C$ inherits the orientation towards the one end. For every vertex $x \in S_k \cap V(H) $ the number of vertices below $x$ is at most $Ck$.

Let $j$ be such that $\bigcup_{i=0}^j S_i$ covers at least half of the vertices in $V(H)$. 
For every $k$ let $I_k$ be the set of leaves of $(V(H)\setminus \bigcup_{i=0}^j S_i) \cap \bigcup_{j=k}^{\infty} S_j$. 
Note that for each $k$ the set $I_k$ is measurable, 
 the sets below the vertices of $I_k$ are disjoint and cover at least half of the vertices of $H$. Also, if 
$x \in I_k$, then there exists a path directed towards the end of length at least $k$ ending at $x$, and the number of vertices below $x$ is at most $Ck$. We can conclude that for every $k$ a positive, at least $\frac{1}{6C}$ proportion of the vertices are in the middle third of a path of length at least $k$ directed towards the end. That is, if we write $R_k$ for the set of the vertices that cut the tree into two components, each of which has at least $k$ elements, then $\nu(R_k)\geq\frac{1}{6C}\nu(V(H))$ for each $k$. If a vertex belongs to infinitely many of the sets $R_k$, the removal of such a vertex splits its component into at least two infinite components, hence it was not one-ended, a contradiction.

\smallskip
\textbf{(ii) $\rightarrow$ (iii)}. We may assume that every component of $G$ is two-ended. We say that a subset $B$ of a component is \textit{bi-infinite} if for every finite subset of the component, $B$ intersects each of the infinite components of the complement of that finite subset. By a \textit{cutset} we mean a finite subset of a component of $G$ whose complement in that component has exactly two infinite components.
In every component, there exists a finite cutset, as the components are two-ended. By possibly gluing such cutsets with the finite components of their complement, we get that in every component, there exists a finite, connected cutset. 
Let $A_n$ denote the union of the components that contain a connected cutset with $n$ vertices
and let $\mathcal{D}_n$ denote the collection of connected cutsets. 
Note that in every component contained in $A_n$ the set $\bigcup\mathcal{D}_n$ is nonempty and in a.e. component contained in $A_n$ the set $\bigcup\mathcal{D}_n$ is bi-infinite, since the components in which it is not bi-infinite admit a selector and thus form a smooth subset, which must be a null set. 
Using \cite[Proposition 4.6]{kst}, we can choose a Borel subset $\mathcal{D}_n' \subseteq \mathcal{D}_n$ that intersects almost every component of $A_n$ in a bi-infinite set and any two elements of $I_n$ are non-adjacent. 

Given a collection $\mathcal{I}$ of cutsets in a component of $G$ such that $\bigcup\mathcal{I}$ is bi-infinite, by a \textit{break} in $\mathcal{I}$ we mean the union of the finite connected components of the complement of $\bigcup\mathcal{I}$ which are adjacent to the same two consecutive elements of $\mathcal{I}$. Note that every cutset in $\mathcal{I}$ is adjacent to exactly two breaks and every component of a break is adjacent to exactly two cutsets in $\mathcal{I}$. 

 We claim that there exists a Borel $\mathcal{C}_n\subseteq\mathcal{D}_n'$ such that $\bigcup\mathcal{C}_n$ is still bi-infinite in every component and the breaks in $\mathcal{C}_n$
 are connected. Indeed, using \cite[Proposition 4.6]{kst}, we can find a Borel refinement $\mathcal{C}_n$ of $\mathcal{D}_n'$ such that $\bigcup\mathcal{C}_n$ is bi-infinite in every component and between every two consecutive elements of $\mathcal{C}_n$ there are either one or two elements of $\mathcal{D}_n'$. We claim that the breaks in $\mathcal{C}_n$ are now connected. Indeed, any two points in such a break $B$ can be connected by a path to any of the cutsets of $\mathcal{D}_n'$ that is contained in $B$. Thus, connectedness of $B$ follows from connectedness of the cutsets in $\mathcal{D}_n'$.

The sets $A_n$ and $\mathcal{C}_n$ are as needed.



\smallskip

\textbf{(iii) $\rightarrow$ (i)}. Consider $n>0$ and the set $A_n$. It suffices to show that $A_n$ has linear growth a.e., for every $n$. 
Consider the measurably bipartite graphing $H_n$ on $V(G)$, where two vertices are adjacent if one is in $\bigcup\mathcal{C}_n$ and the other is in an adjacent component of $A_n\setminus\bigcup\mathcal{C}_n$. It is enough to show that $H_n$ has linear growth. By an \textit{interval} we mean a finite union of consecutive components of $A_n\setminus\bigcup\mathcal{C}_n$. The \textit{length} of an interval $I$, denoted by $l(I)$ is the number of consecutive components of $A_n\setminus\bigcup\mathcal{C}_n$ contained in $I$.  A cutset $C$ is \textit{in the middle of} an interval $I$ if $C$ is adjacent to two components of the interval $I$. Consider a $C>0$ and the family of intervals $$\mathcal{I}_C=\{I: |I|>C\cdot l(I) \}.$$ 

Given $L>0$ consider the subfamily of $\mathcal{I}_C$ consisting of intervals of length at most $L$. It admits a measurable subfamily $\mathcal{I}^L_C$ that covers the same set of vertices, but every vertex at most twice. Indeed, we can construct it using \cite[Proposition 4.5]{kst}, by taking the subfamily of maximal intervals by containment, and iteratively removing an interval if it is contained by the union of two other intervals. 
If a cutset is in the middle of an interval in $\mathcal{I}^L_C$, then there can be in the middle of at most two such intervals. Hence the measure of the vertices in $\bigcup\mathcal{C}_n$ that are in the middle of an interval in
$\mathcal{I}^L_C$  
is at most $\frac{2n}{C}$. Write $M_C$ for the set of vertices in the cutsets that are in the middle of an interval in $\mathcal{I}_C$. As the sets $\mathcal{I}^L_C$ are increasing as $L$ increases, we get that $\nu(M_C)\leq\frac{2n}{C}$. Since this holds for every $C$, 
the measure of $\bigcap_{C>0}M_C$ is zero and so is the measure of the union of components that contain a cutset in $\bigcap_{C>0}M_C$. This implies that for a.e. vertex in 
$\bigcup\mathcal{C}_n$ the balls centered around it have linear growth. Since $\bigcup\mathcal{C}_n$ 
intersects a.e. component of $A_n$, the vertices in $A_n$ have linear growth a.e. for every $n$, so $G$ has a.e. linear growth.

\end{proof}

\begin{proof}[Proof of Lemma \ref{locfin}]
We will first construct a subgraphing $H'$ such that a.e. $H'$ has infinite components, $H'$ is a.e. not two-ended and $\mu(H')$ is finite. 

To this end, we first build a sequence $E_0 \subseteq E_1 \subseteq \ldots \subseteq E(G)$ such that the subgraphing spanned by the edges of $U=\bigcup_{i=1}^{\infty} E_i$ is a.e. locally finite and has infinite components only. Along the way, we make sure that the finite components of the graphing
spanned by $E_i$ have size at least $2^i$. Start with $E_0=\emptyset$. Recall that a set of vertices in a graph is \textit{$k$-discrete} if the distance between any of its two distinct elements is at least $k$. Given $E_i$ add for every component an edge connecting it to another component in such a way that the starting vertices of these edges intersect every finite component of $E_i$ in at most one vertex and in the infinite components of $E_i$ the starting vertices of these edges form a $2^i$-discrete set in $E_i$.  This can be done since such edges exist and the graph spanned by $E_i$ is locally finite, so $E_i$ has a maximal Borel $2^i$-discrete set by \cite[Proposition 4.5]{kst}. Let $E_{i+1}$ be this set of new edges together with $E_i$. 
Since $E_i$ 
has components of size at least $2^i$, the finite components of $E_{i+1}$ are of size at least $2^{i+1}$. Since the graph $E_i$ is pmp, the measure of the starting vertices of the new edges is at most $2^{-i}$. Consequently, we have $\mu(E_{i+1} \setminus E_i) \leq 2^{-i}$, and $\mu(U) \leq 1$. 

Next, we construct $U=U_0 \subseteq U_1\subseteq \ldots \subseteq E(G)$ such that every two-ended component of $U_i$ contains at least $2^i$ components of $U$. Given $U_i$, add for every two-ended component of $U_i$ a Borel set of edges connecting it to another component such that the starting vertices of these edges chosen in a component form a $2^i$-discrete set in $U_i$. Such new edges exist because $G$ is a.e. not two-ended, and we can find them in a Borel way because the graph spanned by $U_i$ is locally finite. Let $U_{i+1}$ be the set of these new edges together with $U_i$. Note that the set of starting vertices of the new edges has measure at most $2^{-i}$, as the graph is pmp, and this implies that $\mu(U_{i+1} \setminus U_i) \leq 2^{-i}$. 

Now, consider the union $H'=\bigcup_{i=0}^\infty U_i$.  We have $\mu(H')\leq 2$ and thus $H'$ is a.e. locally finite. Since every component of $H'$ contains infinitely many components of $U$, the graphing $H'$ has superlinear growth. Proposition \ref{equi} implies that $H'$ is a.e. not two-ended. 

Finally, by \cite[Theorem 2.1]{cgmtd}, $H'$ admits a one-ended spanning subforest $H$ and the subgraphing $H$ is then hyperfinite. 

\end{proof}


\bibliographystyle{amsalpha} 
\bibliography{1} 

\providecommand{\bysame}{\leavevmode\hbox to3em{\hrulefill}\thinspace}
\providecommand{\MR}{\relax\ifhmode\unskip\space\fi MR }
\providecommand{\MRhref}[2]{%
  \href{http://www.ams.org/mathscinet-getitem?mr=#1}{#2}
}
\providecommand{\href}[2]{#2}
\begin{thebibliography}{CGMTD}

\bibitem[Ada90]{adams}
Scott Adams, \emph{Trees and amenable equivalence relations}, Ergodic Theory
  Dynam. Systems \textbf{10} (1990), no.~1, 1--14.

\bibitem[BH21]{benjamini.hutchcroft}
Itai Benjamini and Tom Hutchcroft, \emph{Large, lengthy graphs look locally
  like lines}, Bull. Lond. Math. Soc. \textbf{53} (2021), no.~2, 482--492.

\bibitem[BHTa]{bht1}
Ferenc Bencs, Aranka Hru\v{s}kov\'{a}, and L\'{a}szl\'{o}~M\'{a}rton T\'{o}th,
  \emph{Factor-of-iid balanced orientation of non-amenable graphs}, preprint,
  arXiv:2106.12530.

\bibitem[BHTb]{bht}
\bysame, \emph{Factor-of-iid schreier decorations of lattices in {E}uclidean
  spaces}, preprint, arXiv:2101.12577.

\bibitem[BLPS99]{blps}
I.~Benjamini, R.~Lyons, Y.~Peres, and O.~Schramm, \emph{Group-invariant
  percolation on graphs}, Geom. Funct. Anal. \textbf{9} (1999), no.~1, 29--66.

\bibitem[BT24]{banach-tarski}
S.~Banach and A.~Tarski, \emph{Sur la d\'ecomposition des ensembles de points
  en parties respectivement congruentes}, Fund. Math \textbf{6} (1924), no.~1,
  244--277.

\bibitem[CG]{grabowski.ciesla}
Tomasz Cie\'sla and {\L}ukasz Grabowski, \emph{On random compact sets,
  equidecomposition, and domains of expansion in $\mathbb{R}^3$}, preprint,
  arXiv:2104.01244.

\bibitem[CGMTD]{cgmtd}
Clinton Conley, Damien Gaboriau, Andrew Marks, and Robin Tucker-Drob,
  \emph{One-ended spanning subforests and treeability of groups}, preprint,
  arXiv:2104.07431.

\bibitem[CK13]{conley.kechris}
Clinton~T. Conley and Alexander~S. Kechris, \emph{Measurable chromatic and
  independence numbers for ergodic graphs and group actions}, Groups Geom. Dyn.
  \textbf{7} (2013), no.~1, 127--180.

\bibitem[CL17]{csoka.lippner}
Endre Cs\'{o}ka and Gabor Lippner, \emph{Invariant random perfect matchings in
  {C}ayley graphs}, Groups Geom. Dyn. \textbf{11} (2017), no.~1, 211--243.

\bibitem[CLP16]{clp}
Endre Cs\'{o}ka, G\'{a}bor Lippner, and Oleg Pikhurko, \emph{K{\H o}nig's line
  coloring and {V}izing's theorems for graphings}, Forum Math. Sigma \textbf{4}
  (2016), Paper No. e27, 40.

\bibitem[CM16]{conley.miller.toast}
Clinton~T. Conley and Benjamin~D. Miller, \emph{A bound on measurable chromatic
  numbers of locally finite {B}orel graphs}, Math. Res. Lett. \textbf{23}
  (2016), no.~6, 1633--1644.

\bibitem[CM17]{conley.miller.pm}
\bysame, \emph{Measurable perfect matchings for acyclic locally countable
  {B}orel graphs}, J. Symb. Log. \textbf{82} (2017), no.~1, 258--271.

\bibitem[CS22]{ciesla.sabok}
Tomasz Cie\'{s}la and Marcin Sabok, \emph{Measurable {H}all's theorem for
  actions of abelian groups}, J. Eur. Math. Soc. (JEMS) \textbf{24} (2022),
  no.~8, 2751--2773.

\bibitem[Dri84]{drinfeld}
V.~G. Drinfeld, \emph{Finitely-additive measures on {$S^{2}$} and {$S^{3}$},
  invariant with respect to rotations}, Funktsional. Anal. i Prilozhen.
  \textbf{18} (1984), no.~3, 77.

\bibitem[Ele12]{elek}
G\'{a}bor Elek, \emph{Finite graphs and amenability}, J. Funct. Anal.
  \textbf{263} (2012), no.~9, 2593--2614.

\bibitem[Fre31]{frendenthal}
Hans Freudenthal, \emph{\"{U}ber die {E}nden topologischer {R}\"{a}ume und
  {G}ruppen}, Math. Z. \textbf{33} (1931), no.~1, 692--713.

\bibitem[GJKS15]{gjks.forcing}
S.~Gao, S.~Jackson, E.~Krohne, and B~Seward, \emph{Forcing constructions and
  countable borel equivalence relations}, preprint, 2015, available at
  \texttt{https://itservices.cas.unt.edu/~sgao/pub/pub.html}.

\bibitem[GJKS18]{gjks}
\bysame, \emph{Continuous combinatorics of abelian group actions}, preprint,
  arXiv:1803.03872.

\bibitem[GMP17]{gmp}
\L{}ukasz Grabowski, Andr\'{a}s M\'{a}th\'{e}, and Oleg Pikhurko,
  \emph{Measurable circle squaring}, Ann. of Math. (2) \textbf{185} (2017),
  no.~2, 671--710.

\bibitem[GMP20]{gmp1}
\L{}ukasz Grabowski, Andr\'as M\'ath\'e, and Oleg Pikhurko, \emph{Measurable
  equidecompositions for group actions with an expansion property}, Journal of
  the European Mathematical Society (2020), in press.

\bibitem[GP20]{gp}
Jan Greb\'{\i}k and Oleg Pikhurko, \emph{Measurable versions of {V}izing's
  theorem}, Adv. Math. \textbf{374} (2020), 107378, 40.

\bibitem[HPPS09]{poisson.matching}
Alexander~E. Holroyd, Robin Pemantle, Yuval Peres, and Oded Schramm,
  \emph{Poisson matching}, Ann. Inst. Henri Poincar\'{e} Probab. Stat.
  \textbf{45} (2009), no.~1, 266--287.

\bibitem[JKL02]{jkl}
S.~Jackson, A.~S. Kechris, and A.~Louveau, \emph{Countable {B}orel equivalence
  relations}, J. Math. Log. \textbf{2} (2002), no.~1, 1--80. \MR{1900547}

\bibitem[Kec95]{kechris}
Alexander~S. Kechris, \emph{Classical descriptive set theory}, Graduate Texts
  in Mathematics, vol. 156, Springer-Verlag, New York, 1995. \MR{1321597}

\bibitem[KM04]{kechris.miller}
Alexander~S. Kechris and Benjamin~D. Miller, \emph{Topics in orbit
  equivalence}, Lecture Notes in Mathematics, vol. 1852, Springer-Verlag,
  Berlin, 2004.

\bibitem[KM16]{kechris.marks}
Alexander~S. Kechris and Andrew~S. Marks, \emph{Descriptive graph
  combinatorics}, 2016, preprint.

\bibitem[KNSS02]{klopotowski}
A.~K{\l}opotowski, M.~G. Nadkarni, H.~Sarbadhikari, and S.~M. Srivastava,
  \emph{Sets with doubleton sections, good sets and ergodic theory}, Fund.
  Math. \textbf{173} (2002), no.~2, 133--158.

\bibitem[KST99]{kst}
A.~S. Kechris, S.~Solecki, and S.~Todorcevic, \emph{Borel chromatic numbers},
  Adv. Math. \textbf{141} (1999), no.~1, 1--44.

\bibitem[Kun]{gabor.new}
G\'abor Kun, \emph{The measurable {H}all theorem fails for treeings}, preprint,
  arXiv:2106.02013.

\bibitem[Kun21]{gabor}
\bysame, \emph{On {G}ardner's conjecture}, Combinatorica (2021).

\bibitem[Lac88]{lacz3}
M.~Laczkovich, \emph{Closed sets without measurable matching}, Proc. Amer.
  Math. Soc. \textbf{103} (1988), no.~3, 894--896.

\bibitem[Lac90]{lacz11}
\bysame, \emph{Equidecomposability and discrepancy; a solution of {T}arski's
  circle-squaring problem}, J. Reine Angew. Math. \textbf{404} (1990), 77--117.

\bibitem[Lac92a]{lacz2}
Mikl\'{o}s Laczkovich, \emph{Decomposition of sets with small boundary}, J.
  London Math. Soc. (2) \textbf{46} (1992), no.~1, 58--64.

\bibitem[Lac92b]{lacz.uniform}
\bysame, \emph{Uniformly spread discrete sets in {${\bf R}^d$}}, J. London
  Math. Soc. (2) \textbf{46} (1992), no.~1, 39--57.

\bibitem[Lac96]{laczkovich.dec}
M.~Laczkovich, \emph{Decomposition using measurable functions}, C. R. Acad.
  Sci. Paris S\'{e}r. I Math. \textbf{323} (1996), no.~6, 583--586.

\bibitem[LN11]{lyons.nazarov}
Russell Lyons and Fedor Nazarov, \emph{Perfect matchings as {IID} factors on
  non-amenable groups}, European J. Combin. \textbf{32} (2011), no.~7,
  1115--1125.

\bibitem[Los87]{losert}
Viktor Losert, \emph{On the structure of groups with polynomial growth}, Math.
  Z. \textbf{195} (1987), no.~1, 109--117.

\bibitem[Lov12]{lovasz}
L\'{a}szl\'{o} Lov\'{a}sz, \emph{Large networks and graph limits}, American
  Mathematical Society Colloquium Publications, vol.~60, American Mathematical
  Society, Providence, RI, 2012.

\bibitem[Mar80]{margulis}
G.~A. Margulis, \emph{Some remarks on invariant means}, Monatsh. Math.
  \textbf{90} (1980), no.~3, 233--235.

\bibitem[M{\'{a}}t18]{MatheICM}
Andr\'{a}s M{\'{a}}th\'{e}, \emph{Measurable equidecompositions}, Proceedings
  of the {I}nternational {C}ongress of {M}athematicians---{R}io de {J}aneiro
  2018. {V}ol. {III}. {I}nvited lectures, World Sci. Publ., Hackensack, NJ,
  2018, pp.~1713--1731.

\bibitem[MR]{two-ended}
Babak Miraftab and Tim R\"uhmann, \emph{Two-ended quasi-transitive graphs},
  Discrete Mathematics, Algorithms and Applications,
  https://www.worldscientific.com/doi/abs/10.1142/S1793830922500239.

\bibitem[MU17]{marks-unger}
Andrew~S. Marks and Spencer~T. Unger, \emph{Borel circle squaring}, Ann. of
  Math. (2) \textbf{186} (2017), no.~2, 581--605.

\bibitem[Sch08]{schramm}
Oded Schramm, \emph{Hyperfinite graph limits}, Electron. Res. Announc. Math.
  Sci. \textbf{15} (2008), 17--23.

\bibitem[Sul81]{sullivan}
Dennis Sullivan, \emph{For {$n>3$} there is only one finitely additive
  rotationally invariant measure on the {$n$}-sphere defined on all {L}ebesgue
  measurable subsets}, Bull. Amer. Math. Soc. (N.S.) \textbf{4} (1981), no.~1,
  121--123.

\bibitem[SW79]{scott.wall}
Peter Scott and Terry Wall, \emph{Topological methods in group theory},
  Homological group theory ({P}roc. {S}ympos., {D}urham, 1977), London Math.
  Soc. Lecture Note Ser., vol.~36, Cambridge Univ. Press, Cambridge-New York,
  1979, pp.~137--203.

\bibitem[Tar25]{Tarski.problem}
Alfred Tarski, \emph{Probl\'eme 38}, Fund. Math. \textbf{7} (1925).

\bibitem[Tho22]{thornton2022orienting}
Riley Thornton, \emph{Orienting borel graphs}, Proceedings of the American
  Mathematical Society (2022).

\bibitem[Tim19]{timar}
\'{A}d\'{a}m Tim\'ar, \emph{One-ended spanning trees in amenable unimodular
  graphs}, Electron. Commun. Probab. \textbf{24} ({}2019), Paper No. 72, 12.

\bibitem[Tim21]{timar.new}
\'Ad\'am Tim\'ar, \emph{A factor matching of optimal tail between {P}oisson
  processes}, preprint, arXiv:2106.04524.

\bibitem[Tro84]{trofimov}
V.~I. Trofimov, \emph{Graphs with polynomial growth}, Mat. Sb. (N.S.)
  \textbf{123(165)} (1984), no.~3, 407--421.

\bibitem[TW16]{wagon}
Grzegorz Tomkowicz and Stan Wagon, \emph{The {B}anach-{T}arski paradox}, second
  ed., Encyclopedia of Mathematics and its Applications, vol. 163, Cambridge
  University Press, New York, 2016, With a foreword by Jan Mycielski.

\bibitem[TY16]{tointon.yadin}
Matthew C.~H. Tointon and Ariel Yadin, \emph{Horofunctions on graphs of linear
  growth}, C. R. Math. Acad. Sci. Paris \textbf{354} (2016), no.~12,
  1151--1154.

\bibitem[Vil03]{villani}
C\'{e}dric Villani, \emph{Topics in optimal transportation}, Graduate Studies
  in Mathematics, vol.~58, American Mathematical Society, Providence, RI, 2003.

\bibitem[Weh92]{wehrung}
Friedrich Wehrung, \emph{Injective positively ordered monoids. {I}, {II}}, J.
  Pure Appl. Algebra \textbf{83} (1992), no.~1, 43--82, 83--100.

\bibitem[Wei]{weilacher}
Felix Weilacher, \emph{Borel edge colorings for finite dimensional groups},
  preprint, arXiv:2104.14646.

\end{thebibliography}

\end{document}